\documentclass[11pt,notcite,notref]{article}

\usepackage[T1]{fontenc}
\usepackage[utf8]{inputenc}
\usepackage{amsmath,cases}
\usepackage{amsfonts, amssymb, mathtools}
\usepackage{graphicx}%
\usepackage{placeins}
\usepackage{indentfirst}
\usepackage{caption}
\usepackage{subcaption}
\usepackage{bm}
\usepackage{xcolor}
\usepackage{tikz}
\usetikzlibrary{decorations.pathreplacing,decorations.markings}
\usepackage{comment}
\usepackage{cancel}
\usepackage{empheq}
\allowdisplaybreaks[1]
\usepackage[a4paper, total={6in, 8in}]{geometry}

\numberwithin{equation}{section}

\providecommand{\U}[1]{\protect\rule{.1in}{.1in}}
\newtheorem{theorem}{Theorem}[section]
\newtheorem{corollary}[theorem]{Corollary}

\newtheorem{lemma}[theorem]{Lemma}

\newtheorem{remark}[theorem]{Remark}
\newenvironment{proof}[1][Proof]{\noindent\textbf{#1.} }{\ \rule{0.5em}{0.5em}}

\renewcommand{\div}{\mathrm{div}_{\mathbf{x}}}
\newcommand{\vc}[1]{\mathbf{#1}}

\newcommand{\ignore}[1]{}

\usepackage{authblk}
\AtBeginDocument{}

\date{}
\title{A thin film model for meniscus evolution}
\author[1]{Amrita Ghosh \thanks{ghosh@iam.uni-bonn.de}}
\affil[1]{Institute for Applied Mathematics, University of Bonn \authorcr Endenicher Allee 60, 53115 Bonn, Germany}
\author[1]{Juan J.L. Vel\'azquez \thanks{velazquez@iam.uni-bonn.de}}

\begin{document}
	\maketitle
	
	\begin{abstract}
		In this paper, we discuss a particular model arising from sinking of a rigid solid into a thin film of liquid, i.e. a liquid contained between two solid surfaces and part of the liquid surface is in contact with the air. The liquid is governed by Navier-Stokes equation, while the contact point, i.e. where the gas, liquid and solid meet, is assumed to be given by a constant, non-zero contact angle. We consider a scaling limit of the liquid thickness (lubrication approximation) and the contact angle between the liquid-solid and the liquid-gas interfaces is close to $\pi$.	
		This resulting model is a free boundary problem for the equation $h_t + (h^3h_{xxx})_x = 0$, for which we have $h>0$ at the contact point (different from the usual thin film equation with $h=0$ at the contact point).
		We show that this fourth order quasilinear (non-degenerate) parabolic equation, together with the so-called partial wetting condition at the contact point,
		is well-posed. Furthermore, the contact point in our thin film equation can actually move, contrary to the classical thin film equation for a droplet arising from the no-slip condition. Additionally, we show the global stability of steady state solutions in a periodic setting.
	\end{abstract}
	
	\section{Introduction}

	In this paper, we investigate a particular thin film model, where a rigid solid enters a liquid film (cf. Figure \ref{fig3}), leading to movement of the contact point (where the gas, liquid and solid meet) and the formation of a meniscus, as the initial state is out of equilibrium. This free boundary problem gives rise to the following fourth order non-linear parabolic equation, in one dimension,
	\begin{subequations}
		\begin{align}
			\partial_t h + \partial_x(h^3 \partial^3_xh) = 0 \qquad \text{ in } x>\Lambda, t>0, \label{0}
		\end{align}
	\end{subequations}
	along with the boundary conditions,
	\begin{equation}
		h= g, \qquad \partial_x h =\partial_x g-k, \qquad h\;\partial_x^3h = -\frac{2\Lambda \;\partial_t g}{g^2} \quad \qquad \text{ at } x=\Lambda, t>0. \tag{1.1b}\label{0_BC}
	\end{equation}
	
	\begin{figure}[h!]
		\centering
		\begin{minipage}{.49\textwidth}
			\centering
			\includegraphics[width=.8\linewidth]{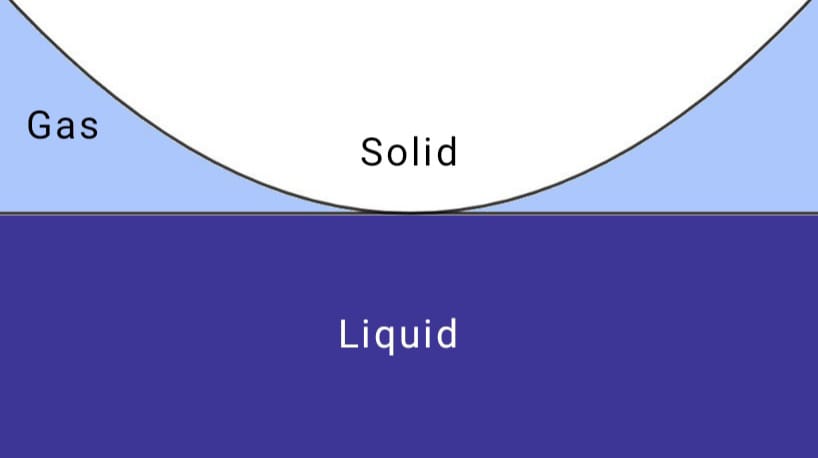}
		\end{minipage}%
		\begin{minipage}{.44\textwidth}
			\centering
			\includegraphics[width=.85\linewidth]{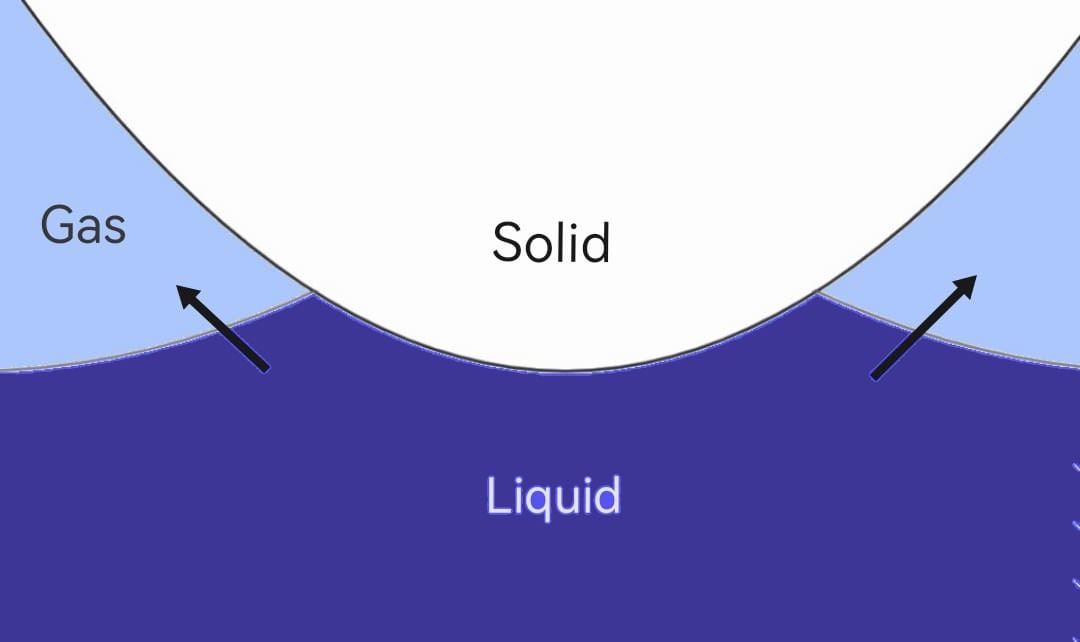}
		\end{minipage}
		\caption{General meniscus formation}
		\label{fig3}
	\end{figure}
	\FloatBarrier
	
	Here,
	\begin{itemize}
		\item $h = h(x,t)$ is the height of the liquid film;
		\item $\Lambda = \Lambda(t)$ is the contact point (which is a free boundary in one dimension);
		\item $k$ is the contact angle (rescaled), i.e. the angle between the liquid-solid and the liquid-gas interfaces; and
		\item $g = g(x,t)$ is the profile of the rigid solid, with a vertically downward motion only.
	\end{itemize}
	The above condition on solid means that $\partial_{xt}g =0$ for all $x,t$.
	
	Formal derivation of the above system is elaborated in Section \ref{S2}. Here the situation in one spatial dimension is studied only, although one may consider general dimension as well (physically relevant cases correspond to $d=1$ or $2$). The study of fluid problems involving free boundary evolution mostly concerns with an interface between two phases (cf. \cite{Solonnikov77}), whereas the interest of the current paper is the situation where three phases of matter meet (typically, liquid-solid-air), generating a \textit{contact line/point} (also known as \textit{triple junction}). The main differences between the model considered here with other classical triple junction models is pointed out later.
	
	\paragraph{Equation (\ref{0}):}	Let us first discuss  equation
	(\ref{0}) which can be viewed as a particular case of the general class of thin film equations
	\begin{equation}
		\label{general_TFE}
		\partial_t h + \partial_x(f(h)\partial_x^3 h) =0 \quad \text{ in } \{h>0\}.
	\end{equation}
	Equations of the form (\ref{general_TFE}) are typically obtained by lubrication approximation of the motion of a two-dimensional viscous liquid droplet, spreading over a solid substrate by the effect of surface tension (cf. \cite{Bertozzi98}). Classical macroscopic fluid mechanics to model such phenomenon imposes the no-slip boundary condition at the liquid-solid interface, that is the velocity of the liquid must be equal to that of the solid substrate in contact with. This corresponds to the case $f(h) = h^3$. However, this condition cannot hold at the moving contact point/triple junction where the solid, liquid and gas, three phases meet. It is known to give rise to a non-integrable singularity at the contact point (cf. \cite{HS}) which further implies that boundary of the film (i.e. the triple junction) cannot move in time (\textit{no slip paradox}). On other hand, Navier slip condition (corresponding to the case $f(h) = h^2$), proposed by Navier \cite{navier}, which allows the liquid to slip over the solid surface (characterized by a slip parameter), resolves such paradox (cf. \cite{ODB} for a discussion in this direction).

	Note that equation (\ref{general_TFE}) can also be viewed as a fourth order version for the elliptic part of the porous medium equation (with different sign)
	\begin{equation*}
		\partial_t h - \partial_x(h^n \partial_xh) =0, \quad n>0.
	\end{equation*}
	However, the standard techniques used for such second order degenerate parabolic equations (such as a maximum principle) are not available for the higher order counterpart which makes the analysis more difficult for the latter (cf. \cite{GKO08}).
	
	\paragraph{Equation (\ref{0_BC}):}	Next let us discuss the boundary conditions in (\ref{0_BC}). The first condition gives the position of the free boundary, whereas the second one determines the slope at the free boundary (contact angle condition). The case where the (equilibrium) contact angle $k$ is given by a non-zero constant, determined by \textit{Young's law} (balance between interfacial energies) (cf. \cite{degennes}), is known as \textit{partial wetting}, whereas the zero contact angle is referred as \textit{complete wetting}. This work takes into account both possibilities. There are also notions of dynamic/apparent contact angle (cf. \cite{hocking92, RenE11}). The third boundary condition comes from a matching condition with the internal region (i.e. the liquid region bounded between solid parts), while in the standard setting of a droplet, a third condition is deduced from matching the velocity of the contact point. This would lead to an over-determined problem for a fourth order operator with fixed boundary, but not for the free boundary problem where the boundary is an unknown.
	
	\textit{State of the art.} All the literature available in the context of the thin film equation for contact angle problems consider either compactly supported initial data (droplet case) or an unbounded support with single free boundary. Thus the standard condition $h=0$ of the vanishing liquid height at the free boundary makes the equation $\partial_t h + \partial_x(h^n \partial^3_x h)=0$ singular/degenerate. Following the pioneering work of Bernis and Friedman \cite{BF90}, several results 
	have been established for (\ref{general_TFE}),  considering different forms of $f$, 
	for example in \cite{GKO08}, \cite{GGKO14}-\cite{gnann18}, \cite{hans11}-\cite{KM15}, \cite{seis18}. 
	We also point out the work \cite{GT} where stability of a global solution for the general contact line problem for the full Stokes equation has been discussed. Furthermore, the works \cite{BL22, DP20, GPSW20, IL21} discuss water wave problems for floating bodies (more generally, wave-structure interactions) which primarily concern the inviscid fluid governed by Euler equation, and driven by the inertial forces; Whereas the dominant force for the system in the current paper is the viscosity, modelled by Navier-Stokes equation. As a result, the water wave models are hyperbolic in nature while the contact angle problem considered in the current work is parabolic. On the other hand, rigorous derivation of thin film approximations of free boundary problems for the Stokes equation has been studied by several people, e.g. in \cite{CM12}.	
	The above references are not claimed to be exhaustive.
	
	The current work is concerned with a different thin film model, deduced from the no-slip boundary condition assumption at the liquid-solid interface. The film height being $h>0$ at the contact point makes our model non-degenerate, unlike the standard thin film equation as discussed above. The result shows well-posedness as well as the possibility of a \textit{moving} contact point. Moreover, no paradox arises in spite that the contact point moves.
	The reason of such behavior in this case is
	that the contact angle is very close to $\pi$. As pointed out by Solonnikov in \cite{solonnikov95}, the paradox gets resolved in the case of contact angle $\pi$. 
	The main contribution of the paper is to bring forward this particular interesting yet simple situation involving the evolution of triple junction and the discussion of different behavior in contrast to the known cases.
	This observation intrigued our curiosity to analyze the model in detail.
	
	\textit{Structure of the paper.} We first deduce a general thin film model describing the above particular situation of meniscus formation under the classical approach in Section \ref{S2}. The arguments used in this derivation are standard. Still we write this derivation in detail, since we did not find any suitable reference for such a model. In Section \ref{S3}, the local well-posedness result is established, for the particular case of a no-slip boundary condition. The system being non-degenerate, standard theory for fourth order parabolic quasilinear equations can be employed. Nevertheless, obtaining suitable regularity (in time) in order to closing the final fixed point argument is not straight-forward (cf. Remark \ref{Rem1}). 
	Finally we obtain the asymptotic behavior of a solution, in the periodic setting, in Section \ref{S4}. By looking at perturbations of a steady state of (\ref{0})-(\ref{0_BC}), existence of a unique global solution is obtained with the help of a priori energy estimates. As long as the initial data remains close (in a certain norm) to the stationary solution, the local solution of the thin film model converges to the steady state for all time. Such stability result of the so-called propagating fronts in approximated version of compressible Navier-Stokes equations have been considered in \cite{DP20}.
	
	\subsection{Main results}
	Our first result concerns a formal derivation of the lubrication model of meniscus formation and contact point evolution as a solid body immerses into a thin film of incompressible, viscous liquid, surrounded by gas. 
	Refer to Section \ref{S2} and beginning of Subsection \ref{S3.2} for the details.
	
	\begin{theorem}
		The following system of thin film equations with moving contact point and fixed contact angle (including the complete wetting case) is obtained from the classical fluid dynamical equations, coupled with rigid solid and gas, under no-slip boundary condition at the liquid-solid interface, as
		\begin{equation}
			\label{7}
			\begin{aligned}
				\partial_{t} h + \partial_{x}( h^3 \partial^3_{x} h ) =0 \qquad &\text{ in } x> \Lambda(t), t>0,\\
				h= g, \quad \partial_{x} h = \partial_x g -k, \quad h \;\partial_{x}^3 h = -2\Lambda\frac{\partial_t g}{ g^2} \qquad &\text{ at } x = \Lambda(t), t>0,\\
				h \to 1 \qquad &\text{ as } \ x\to\infty, t>0,\\
				h =h_0 \qquad &\text{ at } \ t=0, x>\Lambda_0. 
			\end{aligned}
		\end{equation}
		Here $\Lambda_0 = \Lambda(0)$ denotes the initial position of the contact point.
	\end{theorem}

	Next, existence of a local in time solution to (\ref{7}) is deduced.
	
	\begin{theorem}[Heuristic version of Theorem \ref{Th2}]
		\label{Th2.}
		For any initial data $(h_0,\Lambda_0)$, satisfying compatibility conditions, there exists a unique strong solution $(h,\Lambda)$, for short times, of the free boundary problem (\ref{7}).	
	\end{theorem}
	
	A rigorous statement and proof of the above theorem is given in Section \ref{S3}.
	Next 
	we consider the system (\ref{7}) over an interval $(0, 2L)$ (in order to have a bounded domain), symmetric with respect to $x=L$ and with periodic boundary conditions (cf. the beginning of Section \ref{S4} for further discussion on the choice of such domain). Precisely,
	\begin{equation}
		\label{25}
		\begin{aligned}
			\partial_{t} h + \partial_{x}( h^3 \partial^3_{x} h ) =0 \qquad \ &\text{ in } \ x \in (\Lambda, L), t>0,\\
			h= g, \quad \partial_{x} h = \partial_x g -k, \quad h \;\partial_{x}^3 h = 0 \qquad \ &\text{ at } \ x = \Lambda, t>0,\\
			\partial_xh =0 \qquad \ &\text{ at } \ x = L, t>0,\\
			h = h_0 \qquad &\text{ at } \ t=0, x \in (\Lambda_0, L),\\
			h(x,t) = h(x+2L,t), \quad g(x, t) = g(x+2L, t), \qquad & \qquad  x, t>0.
		\end{aligned}
	\end{equation}
	
	\noindent The free interface $h(x,t)$ is defined over the domain $(\Lambda, 2L- \Lambda)$ in this case. Also we assume here that the solid is not moving, i.e. $\partial_t g=0$, leading to null flux boundary condition. Such a solution preserves mass for all time $t>0$ (cf. Section \ref{S4.1}). Then, for a given volume $V_0$, a steady state $(\overline{h}, \overline{\Lambda})$ of (\ref{25}) is characterized by
	\begin{equation*}
		\partial_{x}( \overline{h}^3 \partial^3_{x} \overline{h} ) =0 \ \text{ in } \ x \in (\overline{\Lambda}, L),
	\end{equation*}
	subject to the boundary conditions
	\begin{equation*}
		\overline{h} = g, \ \partial_{x} \overline{h} = \partial_x g-k, \ \partial_{x}^3 \overline{h} = 0 \quad \text{ at } x = \overline{\Lambda} \quad \text{ and } \quad
		\partial_x \overline{h}=0 \quad \text{ at } x=L.
	\end{equation*}
	We refer to (\ref{equil_sol}) for an explicit description of the steady state.
	Stability of this steady state for small perturbation is obtained as below. We refer to Section \ref{S4.5} for rigorous statement and proof.
	
	\begin{theorem}[Heuristic version of Theorem \ref{Th3}]
		\label{Th3.}
		For initial data close enough to the steady state of (\ref{25}), there is a unique global solution $(h, \Lambda) $ of the problem (\ref{25}) which converges to the steady state as $t\to\infty$.
	\end{theorem}
	
	\section{Lubrication approximation}
	\label{S2}
	
	In this section, we elaborate the thin film approximation of the original three phases free boundary problem in consideration.\\
	
	\noindent \textbf{Original model}:
	Let us consider a horizontal, two-dimensional thin film of viscous, incompressible Newtonian liquid, of thickness $H$ at the initial time and a rigid
	solid (locally convex) 
	touching the film 
	(cf. Figure \ref{fig4}). 
	As the solid enters the liquid, the contact point moves and makes an angle $(\pi-\theta)$, $\theta>0$, between the liquid-solid and the liquid-gas interface. The liquid-gas interface is described by $y = h(x,t)$ and the solid bottom is parametrized by $y = g(x,t)$.
	One may think of particular cases, for example $g = H\left( 1-\left( \frac{t}{t_0}\right)^n \right)$ with $n >0$ 
	where $t_0$ is the time scale that characterises the motion of the solid 
	moving down. Here $n=1$ corresponds to the solid moving with constant velocity, while $n=2$ corresponds to the case of constant acceleration. 
	A particular case of wedge-shaped solid (in the form $g(x,t) = \tilde{h}(t)+ c |x|$) has been mentioned in \cite[Fig 6]{shikh2020} in a context of a correct modelling approach describing the creation of a contact angle. 
	The free interface makes an angle $\tilde{\theta}$ with the horizontal plane (cf. Figure \ref{fig4}). Therefore, $\tan \tilde{\theta} =-\partial_xh$ and $\theta = -\arctan (\partial_xh) +\arctan (\partial_xg)$. To obtain the lubrication approximation, we need to assume that the two angles $\theta$ and $\tilde{\theta}$ are very small and of the same order of magnitude, which means that the contact angle $(\pi-\theta)$ is close to $\pi$ in this setting. The expected thin film model is obtained when the small parameter $\theta$ vanishes. 
	Note that the thin film approximation would not be valid if $\tilde{\theta}= O(1)$. The contact points are given by $s_i = (\Lambda_i(t), g(\Lambda_i(t),t))$.
	Finally, we assume the far field condition $h \to H$ as $|x| \to \infty$.
	
	Let $\vc{v}^L = (u,v)$ be the velocity of the liquid, with constant density $\varrho^L$, viscosity $\mu^L$ and pressure $p^L$, governed by the Navier-Stokes equations, while the solid motion is governed by the translational velocity $\vc{v}^S(t)$ only, pointing vertically downward. Let us consider the bottom part of the liquid domain as a reference configuration, i.e. at $y = 0$.
	
	\begin{figure}[h!]
		\centering
			\includegraphics[width=.6\linewidth]{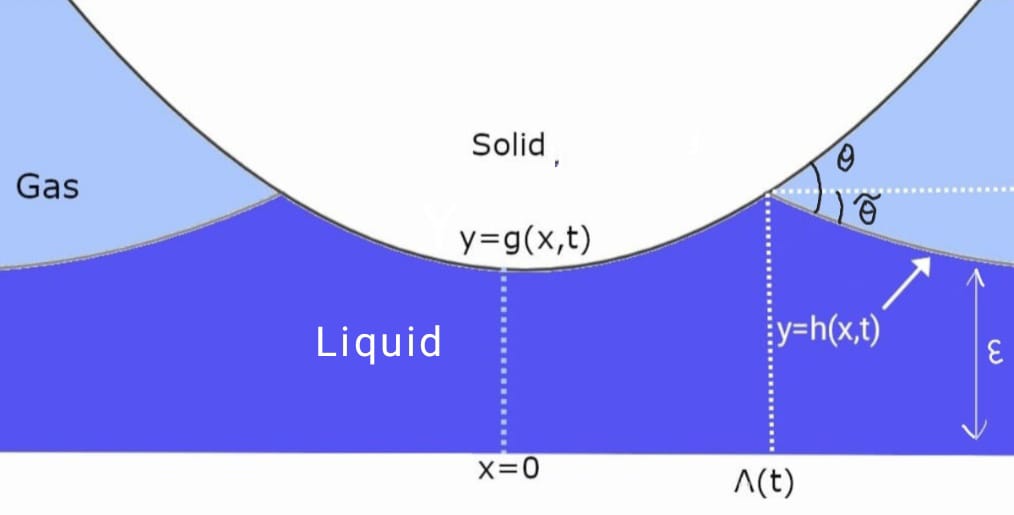}
		\caption{Thin Film approximation with contact points}
		\label{fig4}
	\end{figure}
	\FloatBarrier
	As we assume a symmetric configuration with respect to the $y$-axis (for simplicity), it is enough to analyse the situation for $x > 0$ only. Therefore, in the following, we call the contact point $\Lambda$ instead of $\Lambda_1$.
	
	\paragraph{Further assumptions:} Let the solid bottom $g$ be smooth in a neighborhood of the initial contact point $\Lambda_0$. Moreover, $g>0$ for all $x, t\ge 0$, i.e. the solid never touches the bottom of the liquid domain which may lead to further singularities. Let us mention the work \cite{MR15} where such situation of the solid touching the liquid bottom has been considered. Additionally, we assume that the bottom of the solid fits the liquid at all time $t\ge 0$ which excludes the possibility of formation of air pockets. 
	\\
	
	In the framework of classical fluid mechanics, conditions at the triple junction are considered only at equilibrium, i.e. in the situations where the contact point is not moving with respect to the bulk phases. Furthermore, we consider here the Navier slip condition at the liquid-solid interface, for a general framework, which covers both the no-slip and the full-slip condition. We formulate these conditions below (cf. \cite[Section 2.4]{GNV}).\\
	
	\noindent\textbf{in the liquid:}
	\begin{equation}
		\label{1}
		\div \;\vc{v}^L =0, \qquad \varrho^L\left( \partial_t \vc{v}^L + \vc{v}^L\cdot \nabla_{\mathbf{x}}\vc{v}^L\right) = \mu^L\Delta \vc{v}^L - \nabla_{\mathbf{x}} p^L;
	\end{equation}
	
	\noindent \textbf{at the liquid-gas interface: $y=h(x,t), x>\Lambda(t)$:}
	\begin{equation}
		V_n = \vc{v}^L\cdot \vc{n}_1, \qquad
		\left( p^G - p^L\right) + 2\mu^L\left[ \mathbb{D}\vc{v}^L \cdot\vc{n}_1\right] \cdot \vc{n}_1 + p^s \kappa = 0, \qquad
		\left[ \mathbb{D}\vc{v}^L\cdot\vc{n}_1\right]_\tau = 0; \label{6}
	\end{equation}
	
	\noindent \textbf{at the liquid-solid interface $\{y=0\}\cup \{y= g(x,t), 0<x<\Lambda(t)\}$:}
	\begin{equation}
		\left( \vc{v}^L - \vc{v}^S\right) \cdot \vc{n}_2 =0,\qquad 2 \mu^L\left[  \mathbb{D}\vc{v}^L\cdot\vc{n}_2\right]_\tau = \beta \left( \vc{v}^L_\tau - \vc{v}^S_\tau\right); \label{13}
	\end{equation}
	
	\noindent \textbf{at the triple junction $x=\Lambda(t)$:}
	\begin{equation}
		h =  g,\qquad
		\arctan(\partial_x h) = \arctan(\partial_x g) - \theta. \label{2}
	\end{equation}
	Here $V_n$ is the normal velocity of the free boundary, $\mathbb{D} := \frac{1}{2}(\nabla_{\mathbf{x}} + \nabla_{\mathbf{x}}^T)$ is the symmetric gradient, $p^G$ is the constant pressure of the gas, $p^s$ is the (constant) surface tension, $\kappa$ is the curvature of the free interface, $\vc{n}_i, i = 1,2$ are the unit normal vectors, inward with respect to the liquid domain, and the subscript $\tau$ denotes the tangential component of a vector. We do not distinguish here between the slip-coefficients appearing at the lower and upper liquid-solid interface and denote both of them by $\beta$ only. In principle, they might be different. We further assume for simplicity that the lower boundary $y = 0$ is fixed, i.e. $\vc{v}^S = 0$ at $y = 0$.
	
	Also, we must have the matching condition for the initial data,
	\begin{equation}
		\label{initial_matching}
		h(x,0) = H \qquad \text{ at } t=0.
	\end{equation}
	Further, due to the symmetry of the considered configuration, we assume,
	\begin{equation}
		\label{symmetry_cond}
		u=0 \qquad \text{ on } x=0.
	\end{equation}
	
	Our goal here is to obtain the approximate model where the horizontal length scale of the system is much larger than the vertical length scale and the small angle $\theta$ is comparable to this ratio.
	Precisely, let us introduce the length scale as $L =\frac{ H}{\varepsilon}$ where $\varepsilon\ll 1$. Then the usual scaling of the variables for non-dimensionalization are given by,
	\begin{align*}
		&(\overline{x},  \overline{y}) = \left( \frac{x}{L}, \frac{y}{H}\right)  && \overline{\Lambda} = \frac{1}{L} \Lambda, && \overline{h} = \frac{h}{H}, \overline{g} = \frac{g}{H}, \\[.3cm]
		&\overline{t} = \frac{\sigma}{L \mu^L} t,  && (\overline{u}, \overline{v}) = \left( \frac{\mu^L}{\sigma} u, \frac{\mu^L}{\varepsilon \sigma} v\right) , &&
		\overline{p} =\frac{\varepsilon^2 L}{\sigma} p^L, \\[.3cm]
		&\theta =k \varepsilon, \quad && \overline{\beta} = \frac{\varepsilon L}{\mu^L}\beta, &&
	\end{align*}
	where $\sigma = -\varepsilon^3 p^s$ is the scaled surface tension. All the above scalings are chosen as per the usual approach for lubrication theory. In particular, we assume that $H$ and $\sigma$ are of order $1$, 
	hence 
	\begin{equation*}
		\text{Reynolds number} =
		\varepsilon^3 \overline{\text{Re}}, \qquad \text{ Capillary number} = 
		\varepsilon^3,
	\end{equation*}
	where $\overline{\text{Re}}$ is of order $1$. Here $\overline{\beta}$ is a constant and can take any value in $[0, \infty]$.
	\\
	
	\noindent \textbf{Asymptotic system}:	Under these scalings, we obtain the following system at the leading order of $\varepsilon$ from (\ref{1})-(\ref{2}),
	\begin{subequations}
		\begin{align}
			&\partial_{\overline{x}} \overline{u} + \partial_{\overline{y}} \overline{v} =0, \qquad \partial^2_{\overline{y}} \overline{u} = \partial_{\overline{x}} \overline{p}, \qquad \partial_{\overline{y}} \overline{p}=0, \label{19}\\
			& \hspace{5.38cm}\text{in } (\overline{x}, \overline{y}) \in ( 0, \overline{\Lambda} ) \times ( 0, \overline{g}) \cup ( \overline{\Lambda}, \infty) \times ( 0, \overline{h}), \nonumber\\
			& \nonumber\\
			&\partial_{\overline{t}} \overline{h} + \overline{u} \;\partial_{\overline{x}} \overline{h} -\overline{v} =0, \qquad \overline{p}=- \partial^2_{\overline{x}} \overline{h}, \qquad
			\partial_{\overline{y}} \overline{u}=0, \hspace{0.4cm}\qquad \text{ on } \overline{x}> \overline{\Lambda}, \; \overline{y}=\overline{h},	\label{41}\\
			&\overline{v}=0, \qquad \partial_{\overline{y}} \overline{u} =\overline{\beta} \;\overline{u}, \hspace{6.3cm}\qquad \text{ on } \overline{y}=0, \label{bottomBC}\\
			&\partial_{\overline{x}} \overline{g}\; \overline{u} = \overline{v} - \partial_{\overline{t}} \overline{g}, \qquad \partial_{\overline{y}} \overline{u} + \overline{\beta}\; \overline{u} =0, \hspace{2.8cm}\qquad \text{on } \overline{x} < \overline{\Lambda}, \; \overline{y}= \overline{g}, \label{uppersolidBC}\\
			&\overline{h} = \overline{g}, \qquad
			\partial_{\overline{x}} \overline{h}=\partial_{\overline{x}} \overline{g} -k, \hspace{5.59cm}\qquad \text{ at } \overline{x} = \overline{\Lambda}, \label{11}\\
			&\overline{h} \to 1 \hspace{9cm}\qquad \text{as } \overline{x}\to\infty, \label{11.}
		\end{align}
	\end{subequations}
	The system is complemented by equations (\ref{initial_matching}), (\ref{symmetry_cond}).
	
	Indeed, let us first write the governing equations (\ref{1})-(\ref{2}) in dimensionless form.
	The incompressibility condition $(\ref{1})_1$ remains the same in the new variables, leading to $(\ref{19})_1$. The Navier-Stokes equation $(\ref{1})_2$ in the horizontal component becomes,
	\begin{equation*}
		\frac{1}{\varepsilon}\;	\underset{\mathrm{Re}}{\underbrace{\varrho^L \frac{\sigma}{\mu^L} \frac{H}{\mu^L}}} \left( \partial_{\overline{t}} \overline{u} + \overline{u} \partial_{\overline{x}} \overline{u} + \overline{v} \partial_{\overline{y}} \overline{u}\right) =  \Big( \partial_{\overline{x}}^2\overline{u} + \frac{1}{\varepsilon^2} \partial_{\overline{y}}^2\overline{u}\Big) - \frac{1}{\varepsilon^2} \partial_{\overline{x}} \overline{p}.
	\end{equation*}
	Thus in the limit $\varepsilon \to 0$, both the time derivative and the non-linear term vanish compared to the pressure gradient as the Reynolds number satisfies $ \mathrm{Re} = \varepsilon^3\overline{\text{Re}}$ with $\overline{\text{Re}} = O(1)$ and one obtains $(\ref{19})_2$. Similarly the vertical component of Navier-Stokes equation reduces to $(\ref{19})_3$.
	
	Next we discuss the boundary conditions. At the liquid-gas interface $\overline{y} = \overline{h}(\overline{x}, \overline{t}), \overline{x} > \overline{\Lambda}(\overline{t})$, the unit normal and tangent vectors and the curvature have the following expressions,
	\begin{equation}
		\label{normal_tangent_vector}
		\vc{n}_1 = \frac{\left( \varepsilon \partial_{\overline{x}} \overline{h}, -1\right) }{\sqrt{(1+ \varepsilon^2 |\partial_{\overline{x}} \overline{h}|^2)}}, \quad \vc{\tau} = \frac{\left(1,  \varepsilon \partial_{\overline{x}} \overline{h}\right) }{\sqrt{(1+ \varepsilon^2 |\partial_{\overline{x}} \overline{h}|^2)}} \quad
		\text{ and } \quad \kappa = \frac{\varepsilon \partial^2_{\overline{x}} \overline{h}}{L\left( 1+ \varepsilon^2 |\partial_{\overline{x}} \overline{h}|^2\right) ^{3/2}}.
	\end{equation}
	The kinematic boundary condition $(\ref{6})_1$ reduces to the non-dimensional form $(\ref{41})_1$. Also the rate of strain tensor becomes, in the new variables,
	\begin{equation*}
		\renewcommand{\arraystretch}{1.5}
		\mathbb{D}\vc{v}^L = 
		\frac{1}{2} \begin{bmatrix}
			2\partial_{x} u & \partial_{y} u + \partial_{x} v\\
			\partial_{y} u + \partial_{x} v & 2 \partial_{y} v
		\end{bmatrix}
		=\frac{1}{2} \frac{ \sigma}{L \mu^L}
		\begin{bmatrix}
			2\partial_{\overline{x}} \overline{u} & \frac{1}{\varepsilon}\partial_{\overline{y}} \overline{u} + \varepsilon \partial_{\overline{x}} \overline{v}\\
			\frac{1}{\varepsilon}\partial_{\overline{y}} \overline{u} + \varepsilon\partial_{\overline{x}} \overline{v} & 2 \partial_{\overline{y}} \overline{v}
		\end{bmatrix}.
	\end{equation*}
	Therefore, the normal and the tangential stress on the surface $\overline{y}=\overline{h}(\overline{x}, \overline{t})$ are given by,
	\begin{equation}
		\label{normal_stress}	\renewcommand{\arraystretch}{1.5}
		\left[\mathbb{D}\vc{v}^L\cdot \vc{n}_1\right]\cdot \vc{n}_1 = \frac{\sigma}{L \mu^L}\frac{1}{\left( 1+ \varepsilon^2|\partial_{\overline{x}} \overline{h}|^2\right) } \left( - \partial_{\overline{y}} \overline{u} \;\partial_{\overline{x}} \overline{h} + \partial_{\overline{y}} \overline{v} + O(\varepsilon^2)\right),
	\end{equation}
	and
	\begin{equation}
		\label{tangential_stress}
		\renewcommand{\arraystretch}{1.5}
		\left[ \mathbb{D}\vc{v}^L\cdot \vc{n}_1\right]\cdot \vc{\tau} = \frac{\sigma}{L \mu^L}\frac{1}{2}\frac{1}{\left( 1+ \varepsilon^2|\partial_{\overline{x}} \overline{h}|^2\right) } \left(- \frac{1}{\varepsilon} \partial_{\overline{y}} \overline{u} + O(\varepsilon)\right) .
	\end{equation}
	\\
	Then, equation $(\ref{6})_2$ becomes, with the help of (\ref{normal_tangent_vector}) and (\ref{normal_stress}),
	\begin{equation*}
		p^G - \frac{ \sigma}{\varepsilon^2 L} \overline{p} + 2 \frac{ \sigma}{L}\frac{1}{\left( 1+ \varepsilon^2|\partial_{\overline{x}} \overline{h}|^2\right) } \left( - \partial_{\overline{y}} \overline{u} \;\partial_{\overline{x}} \overline{h} + \partial_{\overline{y}} \overline{v} + O(\varepsilon^2)\right)- \frac{\sigma}{\varepsilon^3}  \frac{\varepsilon \;\partial^2_{\overline{x}} \overline{h}}{L\left( 1+ \varepsilon^2 |\partial_{\overline{x}} \overline{h}|^2\right) ^{3/2}} =0.
	\end{equation*}
	Due to the assumption that $\sigma$ is of order $1$, one obtains in the limit $(\ref{41})_2$. For equation $(\ref{6})_3$, one gets $(\ref{41})_3$ using the expression (\ref{tangential_stress}).
	
	At the bottom liquid-solid interface, with $\vc{n}_2 = (0,1), \tau = (1,0)$ at $y = 0$, the conditions (\ref{13}) read as,
	\begin{equation*}
		v=0, \qquad \mu^L (\partial_yu+ \partial_xv)= \beta u \qquad \text{ at } y=0,
	\end{equation*}
	which then convert into (\ref{bottomBC}).
	
	Similarly, at the upper liquid-solid interface, the normal and tangent vectors being,
	\begin{equation*}
		\vc{n}_2 = \frac{(\partial_x g,-1)}{\sqrt{1+ |\partial_xg|^2}}, \qquad \tau = \frac{(1,\partial_x g)}{\sqrt{1+|\partial_x g|^2}},
	\end{equation*}
	and $\vc{v}^S = (0,\partial_tg)$, the boundary conditions (\ref{13}) transform into the non-dimensional form (\ref{uppersolidBC}).
	
	The conditions at the contact point (\ref{2}) transform into (\ref{11}).\\
	
	At this point, we assume the particular case of no-slip boundary condition, i.e. $\frac{1}{\overline{\beta}} =0$
	for which the subsequent analysis will be done. Hence, the system (\ref{19})-(\ref{11.}) becomes, omitting the bar here onwards,
	\begin{subequations}
		\begin{align}
			&\partial_{x} u + \partial_{y} v =0, \qquad \partial^2_{y} u = \partial_{x} p, \qquad \partial_{y} p=0, \label{19.}\\
			& \hspace{5.38cm}\text{in } (x, y) \in ( 0, \Lambda ) \times ( 0, g) \cup ( \Lambda, \infty) \times ( 0, h), \nonumber\\
			& \nonumber\\
			&\partial_{t} h + u \;\partial_{x} h -v =0, \qquad p=- \partial^2_{x} h, \qquad
			\partial_{y} u=0, \hspace{0.4cm}\qquad \text{ on } x> \Lambda, \; y=h,	\label{41.}\\
			&v=0, \qquad u =0, \hspace{6.3cm}\qquad \text{ on } y=0, \label{bottomBC.}\\
			&\partial_{x} g\; u = v - \partial_{t} g, \qquad u =0, \hspace{2.8cm}\qquad \text{on } x < \Lambda, \; y= g, \label{uppersolidBC.}\\
			&h = g, \qquad
			\partial_{x} h=\partial_{x} g -k, \hspace{5.59cm}\qquad \text{ at } x = \Lambda, \label{11...}\\
			&h \to 1 \hspace{9cm}\qquad \text{as } x\to\infty, \label{11..}
		\end{align}
	\end{subequations}

	\noindent \textbf{Thin film equations}: The thin film model can now be derived from (\ref{19.})-(\ref{11..}), together with the boundary conditions and initial data as,
	\begin{subequations}
		\label{24}
		\begin{align}
			\partial_{t} h +\frac{1}{3} \partial_{x}\Big(  h^3 \partial^3_{x} h \Big)  =0, \qquad &\text{ in } x> \Lambda, \label{thin_film_eqn}\\
			h= g, \qquad
			\partial_{x} h = \partial_x g -k, \qquad &\text{ at } x = \Lambda, \label{contact_line} \\
			h \to 1, \qquad &\text{ as } \ x\to\infty,\\
			h =1, \qquad &\text{ at } \ t=0. \label{IC}
		\end{align}
	\end{subequations}
	The deduction of the above system is explained in the following. Integrating $(\ref{19.})_2$ twice gives the profile of the horizontal velocity,
	\begin{equation}
		\label{hori_vel}
		u(x,y) = \frac{1}{2}\partial_x p \;y^2+ A(x,t) y + B(x,t), \qquad x>0,
	\end{equation}
	where the constants $A, B$ are to be determined. The condition at the lower boundary $(\ref{bottomBC.})_2$ implies,
	\begin{equation}
		\label{5}
		B =0, \qquad x>0.
	\end{equation}
	
	\noindent \textbf{Exterior region}:
	Also the condition at the free boundary $(\ref{41.})_3$ yields,
	\begin{equation}
		\label{3}
		A = -\partial_x p \; h, \qquad x> \Lambda.
	\end{equation}
	Further, as the pressure is independent of $y$ due to $(\ref{19.})_3$, the horizontal velocity (\ref{hori_vel}) becomes, together with (\ref{3}) and $(\ref{41.})_2$,
	\begin{equation}
		\label{4}
		u = \frac{1}{2} \partial^3_x h\; (2h-y)y , \qquad x>\Lambda.
	\end{equation}
	Next $(\ref{19.})_1$ and $(\ref{bottomBC.})_1$ give,
	\begin{equation*}
		v|_{y=h} = -\int\displaylimits_0^h {\partial_x u \;\mathrm{d}y}, \qquad x> \Lambda.
	\end{equation*}
	Thus from $(\ref{41.})_1$, we obtain,
	\begin{equation*}
		\partial_{t} h+ \partial_{x} \Big( \int\displaylimits_0^{h}{u \;\mathrm{d}y}\Big) =0, \qquad x>\Lambda.
	\end{equation*}
	Substituting the velocity profile (\ref{4}) into the above relation finally gives the thin film equation (\ref{thin_film_eqn}). The contact line condition (\ref{contact_line}) is nothing but (\ref{11...}) and the initial data (\ref{IC}) follows from (\ref{initial_matching}).\\
	
	\noindent \textbf{Interior region}:	In order to determine the dynamics of the contact point fully, we need to prescribe sufficient conditions in the inner region $x <\Lambda$  as well. To do so, together with the condition (\ref{5}), the condition at the liquid-solid interface $(\ref{uppersolidBC.})_2$ for $x <\Lambda, y = g$, which reads as, due to (\ref{hori_vel}),
	\begin{equation*}
		\frac{1}{2} \partial_{x} p \; g^2 + A \; g =0,
	\end{equation*}
	gives a relation between $\partial_xp$ and $A$ for $x\in(0, \Lambda)$, that is
	\begin{equation}
		\label{42}
		\partial_{x} p=- \frac{2 A}{ g}.
	\end{equation}
	Also, the incompressibility condition $(\ref{19.})_1$ gives an expression for the normal velocity for $x\in \left( 0,\Lambda\right)$,
	\begin{equation*}
		v\arrowvert_{y=  g} = - \int\displaylimits_0^{\tilde{h}}{\partial_{x} u \;\mathrm{d}y} =- g \left[ \frac{1}{6} \partial^2_{x} p \;g^2 + \frac{1}{2}\partial_{x} A \; g \right] .
	\end{equation*}
	Thus the condition $(\ref{uppersolidBC.})_1$ implies, for $x\in (0,\Lambda)$, together with (\ref{hori_vel}),
	\begin{equation*}
		\begin{aligned}
			& \partial_x g \left[ \frac{1}{2} \partial_{x} p \; g^2 + A g \right] + \partial_{t}g+g \left[ \frac{1}{6} \partial^2_{x} p \; g^2 + \frac{1}{2}\partial_{x} A \;g \right]=0.
		\end{aligned}
	\end{equation*}
	The above equation together with the relation (\ref{42}) yields an ODE for $A$ with rational coefficients,
	\begin{equation}
		\label{ODE}
		\partial_{x} A +2\frac{\partial_x g }{ g } \;A  + 6 \frac{\partial_{t} g} { g^2} =0, \qquad x< \Lambda.
	\end{equation}
	Also, due to the symmetry assumption (\ref{symmetry_cond}), we have the boundary condition
	\begin{equation*}
		A(x, t)=0 \qquad \text{ at } x=0.
	\end{equation*}
	Therefore, the ODE (\ref{ODE}) determines $A$ and in turn $\partial_xp$ in the inner region $x\in(0, \Lambda)$. Here we assume that the horizontal velocity $u$ is continuous, thus the relation (\ref{hori_vel}) holds at $x = 0$ as well. 
	
	Lastly, we would like to determine a relation between $A(\Lambda^-,t)$ and $\partial^3_x h(\Lambda^+,t)$ imposing continuity of the mass flux through $x=\Lambda$. Notice that the distribution of velocity changes at $x=\Lambda$. There seems to be a boundary layer structure at $x=\Lambda$ that seems interesting which we will not examine in detail in this paper.  We recall that the horizontal velocity (\ref{hori_vel}) can be expressed as, due the relations (\ref{5}) leading to $B=0$ for $\beta = \infty$ and (\ref{42}) giving $\partial_x p$,
	\begin{equation*}
		u = \frac{1}{2}\partial_x p \;y^2+ A y = A y\left( 1-\frac{y}{g}\right), \qquad 0< x<\Lambda.
	\end{equation*}		
	Therefore, the flux from $x<\Lambda$ is given by, since $h|_{x=\Lambda}=g|_{x=\Lambda}$,
	\begin{equation}
		\label{39}
		\int\displaylimits_0^{h(\Lambda,t)}{u \;\mathrm{d}y} = A \int\displaylimits_0^{h(\Lambda,t)} {y \left( 1-\frac{y}{g(\Lambda,t)}\right) \mathrm{d}y} = \frac{1}{6}A \; h^2.
	\end{equation} 
	Now we compute the flux from $x>\Lambda$. We have, from (\ref{4}), for $\beta =\infty$,
	\begin{equation}
		\label{40}
		\int\displaylimits_0^{h(\Lambda,t)}{u \;\mathrm{d}y}
		= \frac{1}{2} \partial_x^3 h \int\displaylimits_0^{h(\Lambda,t)} {(2h -y)y \;\mathrm{d}y}
		= \frac{1}{3} \partial_x^3 h \; h^3.
	\end{equation}
	Equating the two fluxes (\ref{39}) and (\ref{40}), 
	we obtain,
	\begin{equation}
		\label{extra_cond}
		A = 3 h \;\partial_{x}^3 h \qquad \text{ at } x = \Lambda.
	\end{equation}
	
	This above relation (\ref{extra_cond}), together with the thin film equations (\ref{thin_film_eqn})-(\ref{IC}) determines fully the dynamics and the position of the contact point, under general slip boundary condition at the liquid-solid interface. This thin film system is described for a particular case of wedge-shaped solid, without its detailed derivation, in \cite[Section 3.3]{GNV}.
	
	\begin{remark}
		One can obtain more general system for any $\beta\in [0, \infty)$, by using (\ref{19})-(\ref{11.}) instead of (\ref{19.})-(\ref{11..}), as stated below.
		\begin{subequations}
			\begin{align*}
				\partial_{t} h + \partial_{x}\Big(  \Big( \frac{h}{3} + \frac{1}{\beta} \Big)h^2 \partial^3_{x} h \Big)  =0, \qquad &\text{ in } x> \Lambda,\\
				h= g, \qquad
				\partial_{x} h = \partial_x g -k, \qquad &\text{ at } x = \Lambda, \\
				h \to 1, \qquad &\text{ as } \ x\to\infty,\\
				h =1, \qquad &\text{ at } \ t=0,
			\end{align*}
		\end{subequations}
		together with a condition similar to (\ref{extra_cond}). Here $A$ satisfies the ODE
		\begin{equation*}
			\label{ODE.}
			\partial_{x} A +r_1(x, t) \;A  + r_2(x, t) =0, \quad 	A(0, t)=0, \qquad x< \Lambda,
		\end{equation*}
		where
		\begin{equation*}
			r_1(x, t) = \frac{\partial_x g \left(\frac{1}{\beta}+ \frac{1}{3} g\right)}{ g \left( \frac{1}{\beta} +\frac{1}{6} g\right) }, \quad r_2(x, t) = \frac{\partial_{t} g} { g \left( \frac{1}{ \beta} +\frac{1}{6} g\right)}.
		\end{equation*}
	\end{remark}

	\section{Local well-posedness for no-slip condition}
	\label{S3}
	
	\subsection{Notations and functional settings}
	\label{S3.1}
	
	Let $C_0[0, \infty)$ denote the Banach space of continuous functions on $[0, \infty)$ vanishing at infinity, together with the supremum norm on $[0, \infty)$. The space $C^k[0, \infty), k\in\mathbb{N}$ of functions with $k$ times continuous and bounded derivatives is endowed with the standard norm
	\begin{equation*}
		\|v\|_{C^k[0,\infty)} := \sum_{m=0}^{k}\|\partial^m v\|_{C[0,\infty)}.
	\end{equation*}
	We write $C[0, \infty) \equiv C^0[0, \infty)$ for the space of all continuous, bounded functions on $[0,\infty)$. 
	The Banach space of $\rho$-H\"{o}lder continuous functions $C^{k+\rho} [0, \infty)$ for $k \in \mathbb{N}_0 := \mathbb{N}\cup \{0\}, \rho\in(0,1)$, is defined by
	\begin{equation*}
		C^{k+\rho} [0, \infty):= \{v\in C^k[0,\infty): [\partial^kv]_{\rho}:= \sup_{\underset{x\neq y}{x,y\in [0,\infty)}} \frac{|\partial^kv(x)-\partial^kv(y)|}{|x-y|^\rho} < \infty\},
	\end{equation*}
	together with the norm
	\begin{equation*}
		\|v\|_{C^{k+\rho} [0, \infty)} := \|v\|_{C^k[0, \infty)} + [\partial^kv]_{\rho}.
	\end{equation*}
	For $\rho = 1$, the above notions coincide with the class of Lipschitz functions which we denote by $C^{1-} [0, \infty)$.
	Sometimes we omit the underlying space $[0, \infty)$ below which should not cause any confusion.

	\subsection{Reduction to a fixed domain}
	\label{S3.2}
	Recall that we are interested to study the specific situation of the no-slip boundary condition at the liquid-solid interface
and will show the local and global well-posedness for this problem. As mentioned in the introduction, such result is different from the standard analysis for the thin film equation for droplet with no-slip condition.

Using the time scaling $t\to 3t, t_0 \to 3t_0$ in order to avoid carrying the fraction $1/3$ in the rest of the paper, the thin film model (\ref{thin_film_eqn})-(\ref{IC}) reduces to,
\begin{align*}
	\partial_{t} h + \partial_{x}( h^3 \partial^3_{x} h )  =0 \qquad &\text{ in } x> \Lambda, t>0,\\
	h=g, \qquad
	\partial_{x} h = \partial_xg -k, \qquad h \;\partial_{x}^3 h = \tfrac{1}{3}A \qquad &\text{ at } x = \Lambda, t>0,\\
	h \to 1 \qquad &\text{ as } \ x\to\infty, t>0,\\
	h =1 \qquad &\text{ at } \ t=0, x>\Lambda_0,
\end{align*}
where $A$ solves the following ODE in $(0, \Lambda)$ (cf. (\ref{ODE})),
\begin{equation}
	\label{A}
	\partial_x A + 2\frac{\partial_x g}{g}A + 6\frac{\partial_t g}{g^2}=0, \qquad A(0,t)=0.
\end{equation}
Solving (\ref{A}), together with the fact that $\partial_{xt} g =0$ for all $x, t>0$ since the solid is having only vertically downward motion, gives
\begin{equation*}
	A(x,t) = -6x\frac{\partial_t g}{g^2}.
\end{equation*}
Indeed, with this choice of $A$, we have
\begin{equation*}
	\partial_x A = -6 \frac{\partial_t g}{g^2} -6x \frac{\partial_{xt} g}{g^2} + 12x \frac{\partial_t g \;\partial_x g}{g^3},
\end{equation*}
which, together with $\partial_{xt} g =0$, satisfies (\ref{A}).
Thus, we obtain the system (\ref{7}). Recall that we have $g>0$ for all $x,t\ge 0$ by assumption.

Without loss of generality, we assume $\Lambda(0) = 0$.
It is natural to look for classical solutions, in particular $\Lambda\in C^1$ that satisfy
\begin{equation}
	\label{18}
	|\Lambda(t)|\le \delta^2 \qquad \text{ and } \quad |\dot{\Lambda}(t)|\le C \qquad \text{ for } \ t\in [0, T],
\end{equation}
for some $\delta>0$ small and some $T>0$. We will prove in the following that these classical solutions exist. In order to do that, after transforming the moving domain problem (\ref{7}) to a fixed domain (cf. (\ref{10})-(\ref{12})) by some change of variables, we show in Theorem \ref{Th4} the existence of $\Lambda$ by means of a fixed point argument.

For such a $\delta$ fixed, let us now consider a cut-off function $\xi_\delta\in C_c^\infty(\mathbb{R})$ i.e.
\begin{equation*}
	\xi_\delta(s) =
	\begin{cases}
		1, & \quad \text{ if } \ 0\le s\le \delta,\\
		0, & \quad \text{ if } \ s\ge 2\delta,	
	\end{cases}
	\qquad \text{ such that } \quad |\xi'_\delta| \le \frac{C}{\delta},
\end{equation*}
and the following bijection
\begin{equation*}
	Q_\Lambda(x,t) = (x-\Lambda(t))\xi_\delta(x-\Lambda(t)) + x (1-\xi_\delta(x-\Lambda(t)))
\end{equation*}
which transforms the moving domain $(\Lambda(t), \infty)$ to $(0,\infty)$.
Denoting by
\begin{equation*}
	\overline{x} = Q_\Lambda(x,t) \qquad \text{ and } \qquad H(\overline{x}, t) = h(x,t),
\end{equation*}
one can compute the derivatives in the new coordinate as,
\begin{equation*}
	\partial_t h = \partial_t H + \dot{\Lambda}(\Lambda \xi'_\delta - \xi_\delta)\partial_{\overline{x}} H, \qquad \partial_x h = (1-\Lambda \xi'_\delta) \partial_{\overline{x}} H.
\end{equation*}
Therefore, the free boundary problem (\ref{7}) reduces to a fixed domain as
\begin{equation}
	\label{8}
	\begin{aligned}
		\partial_{t} H - \dot{\Lambda}(t) \xi_\delta \; \partial_{\overline{x}}H + (1-\Lambda\xi'_\delta)^4 \; \partial_{\overline{x}}( H^3 \partial^3_{\overline{x}} H ) + F_1= 0 \qquad &\text{ in } \overline{x}> 0,\\
		H= \psi_1, \quad \partial_{\overline{x}} H = \psi_2, \quad \partial_{\overline{x}}^3 H = \psi_3 \qquad &\text{ at } \overline{x} = 0,\\
		H \to 1 \qquad &\text{ as } \ \overline{x}\to\infty,\\
		H(\overline{x}, 0) = H_0\equiv h_0(\overline{x}+ \Lambda(0)) \to 1 \qquad &\text{ as } \ \overline{x}\to\infty,
	\end{aligned}
\end{equation}
where
\begin{equation*}
	F_1 \equiv F_1(\dot{\Lambda}, \xi'_\delta, H) \quad \text{ with \ supp} (F_1) \subset [\delta, 2\delta], 
\end{equation*}
and
\begin{equation}
	\label{23}
	\psi_1(\Lambda, t) \equiv g(\Lambda,t), \qquad \psi_2 (\Lambda,t) \equiv \partial_{x} g(\Lambda,t) -k, \qquad \psi_3(\Lambda,t) \equiv -2\Lambda\frac{\partial_t g}{ g^3}(\Lambda,t).
\end{equation}
While it is common to use change of variables with respect to Lagrangian coordinates reformulating a free boundary problem to a fixed domain, the concerned domain in this work being a simple one in one dimension, one can explicitly write down one such transformation (e.g. $Q_\Lambda$ as defined above) which solves the purpose.

Next we use the following transformation in order to lift the boundary conditions and the far field condition,
\begin{equation}
	\label{transformation}
	\begin{aligned}
		H(\overline{x}, t) &= [ \psi_2 (\Lambda,t)\;\overline{x} + \psi_3(\Lambda,t)\; \overline{x}^3]\xi_\delta(\overline{x}) + U(\overline{x}, t) + (1-\xi_\delta(\overline{x}))\\
		&=: a(\overline{x}, t, \Lambda)\xi_\delta(\overline{x}) + U(\overline{x}, t) + (1-\xi_\delta(\overline{x})),
	\end{aligned} 
\end{equation}
where $|U(\overline{x}, t)|\to 0$ as $\overline{x}\to 0$. Observe that  $a(\overline{x},t,\Lambda)>0$ for $\overline{x}>0,t\in (0,T)$ for suitably chosen $\delta$. Then the system (\ref{8}) becomes, omitting the bar over $x$, 
\begin{equation}
	\label{10}
	\begin{aligned}
		\partial_{t} U - \dot{\Lambda} (\psi_2\; \xi_\delta + \partial_xU)\xi_\delta + (1-\Lambda \xi'_\delta)^4\; \partial_{x}( (a\xi_\delta + 1-\xi_\delta + U)^3 \partial^3_{x} U ) = F_2 + F_3 \quad &\text{ in } x> 0,\\
		\partial_{x} U =0, \qquad \partial_{x}^3 U = 0 \quad &\text{ at } x = 0,\\
		U \to 0 \quad &\text{ as } \ x\to\infty,\\
		U(x, 0) =U_0\equiv h_0(x) - a(x,0) \quad&\text{ for } \ x>0,
	\end{aligned}
\end{equation}
together with
\begin{equation}
	\label{12}
	U = \psi_1(\Lambda,t) \qquad \text{ at } x=0, t>0,
\end{equation}
where
\begin{equation}
	\label{F_2}
	F_2 \equiv F_2(\dot{\Lambda}, \xi'_\delta, U) \quad \text{ with \ supp} (F_2) \subset [\delta, 2\delta], 
\end{equation}
and
\begin{equation}
	\label{F_3}
	F_3 = 6(1-\Lambda \xi'_\delta)^4 \;\partial_x ((a\xi_\delta + (1-\xi_\delta) +U)^3 \;\psi_3 \;x^3\;\xi_\delta)  - \partial_t a\;\xi_\delta + \dot{\Lambda}\; 3 x^2 \psi_3\;\xi^2_\delta.
\end{equation}
The above transformation is useful for splitting the full system as a Cauchy problem (\ref{10}) and a fixed point map (\ref{12}).
The general theory for quasilinear parabolic problems can be applied to treat the fourth order system (\ref{10}) which in turn give the existence of a solution for the full problem (\ref{8}).

\begin{remark}
	\label{Rem1}
	As can be seen from (\ref{12}), $\Lambda$ and $U$ must have the same regularity in time; On the other hand, it is not obvious or immediate to obtain the same time regularity from the parabolic system (\ref{10}) if one uses standard Sobolev spaces. Hence we chose to use H\"{o}lder spaces for the solution since it is important not to lose (trace) regularity in order to perform the fixed point argument in (\ref{12}).
\end{remark}

\subsection{Proof of the main result}
\label{S3.3}
We follow the standard approach of finding a solution of (\ref{10}), with suitable estimates, for a prescribed motion $\Lambda$ and then perform a fixed point argument for (\ref{12}).

The idea is to formulate the solution of (\ref{10}) in terms of a convolution with the Green function for the linear problem. The existence result for the linear problem is stated as below.

\begin{lemma}
	\label{Lem4}
	Let $\Lambda\in C^{1+\frac{\alpha}{4}}[0,T]$, $U_0\in C^{4+\alpha}[0,\infty)\cap C_0[0,\infty)$ where $\alpha \in (0,1)$ and $g$ be smooth enough. There exists a fundamental solution $G(x,t,y,\tau)$ of the linear boundary value problem
	\begin{equation}
		\label{9}
		\begin{aligned}
			\partial_{t} U^*+ (1-\Lambda \xi'_\delta)^4\; \partial_{x}( (a\xi_\delta + (1-\xi_\delta))^3 \partial^3_{x} U^* ) + \dot{\Lambda}\; \partial_xU^* \;\xi_\delta =0 \quad &\text{ in } \ x> 0,\\
			\partial_{x} U^* =0, \qquad \partial_{x}^3 U^* = 0 \quad &\text{ at } \ x = 0,\\
			U^* \to 0 \quad &\text{ as } \ x\to\infty,\\
			U^*(x, 0) =U_0 \quad&\text{ for } \ x>0,
		\end{aligned}
	\end{equation}
	satisfying the estimate
	\begin{equation}
		\label{Green_estimate}
		|\partial^m G(x,t, y,\tau)|\le c_m (t-\tau)^{-\frac{m+1}{4}} \exp\left\lbrace -c \left( \tfrac{|x-y|}{(t-\tau)^{1/4}}\right) ^{4/3} \right\rbrace , \quad t\in [0,T] ,
	\end{equation}
	where the above partial derivative $\partial^m G$ is taken with respect to $x$.
\end{lemma}

\begin{proof}
	The existence of such $G$ follows from \cite[Chapter IV.2, Theorem 3.4]{Eidelman69}.
	As per the notations in \cite{Eidelman69}, for our system (\ref{36}), $N=1=n, b=2, r_1 =1, r_2 =3$ and the boundary conditions read as $B\equiv (\partial_x, \partial_{x}^3)u|_{x=0}=0$.
	The Cauchy problem (\ref{9}) can be reduced to a problem with zero initial condition by considering a function $(U^* - U_0)$ since $U_0\in C^{4+\alpha}[0,\infty)$.
	In $(\ref{9})_1$, the leading order coefficient  $a_4(x,t) =(1-\Lambda \xi'_\delta)^4 \;(a\xi_\delta + (1-\xi_\delta))^3$ is H\"older continuous in both $t\in (0,T)$ and $x\in (0,\infty)$, due to the expressions (\ref{23}), (\ref{transformation}) and the assumptions on $\Lambda, g$. Furthermore, the other coefficients $a_3(x,t) = (1-\Lambda \xi'_\delta)^4 \;\partial_x(a\xi_\delta + (1-\xi_\delta))^3$ and $a_1(x,t) = \dot{\Lambda}\xi_\delta$ are H\"older continuous in $x$ as well. Therefore, all the conditions of \cite[Chapter IV.2, Theorem 3.4]{Eidelman69} are satisfied and hence the existence and estimates of a fundamental solution.
	\hfill
\end{proof}

Next one obtains the following existence result for (\ref{10}), for a given fixed $\Lambda$. 
\begin{theorem}
	Let $U_0\in C^{4+\alpha}[0,\infty)\cap C_0[0,\infty)$ where $\alpha \in (0,1)$, satisfying compatibility conditions
	\begin{equation*}
		\partial_x U_0 = \partial_x^3 U_0 =0 \quad \text{ at } \quad x=0.
	\end{equation*}
	Then for any $\delta>0$ and $\Lambda\in C^{1+\frac{\alpha}{4}}[0, T]$ where $T$ is as in (\ref{18}), there exists $T_0 \in (0,T)$ such that (\ref{10}) has a unique solution $U(x,t)$ belonging to $C^{1+\frac{\alpha}{4}}((0,T_0), C^{4+\alpha}[0,\infty))$.
	
	The time interval $T_0$ depends on the upper bounds of the H\"older constants of the coefficients of (\ref{10}).
\end{theorem}

\begin{proof}
	The solution of (\ref{10}) can be expressed by the following integro-differential equation
	\begin{equation}
			U(x,t) = \int\displaylimits_0^t \int\displaylimits_0^\infty G(x,t,y,\tau) \left[ \dot{\Lambda} \psi_2\; \xi^2_\delta + F_2 + F_3 + F_4 \right](y,\tau)\;\mathrm{d}y \;\mathrm{d}\tau + \int\displaylimits_0^\infty G(x,t,y,0) U_0(y)\;\mathrm{d}y, \label{14}
	\end{equation}
	where $G$ is the fundamental solution of the linear problem (\ref{9}), given in Lemma \ref{Lem4}, $F_2, F_3$ are defined in (\ref{F_2}), (\ref{F_3}) and
	\begin{equation}
		\label{F_4}
		F_4 := - (1-\Lambda \xi'_\delta)^4\; \partial_x(\{U^3 +3 U (a\xi_\delta + 1-\xi_\delta)(U + (a\xi_\delta + 1-\xi_\delta))\}\;\partial_x^3U).
	\end{equation}
	Solving the above equation is standard, for example one may refer to \cite[Chapter III.4, Theorem 8.3]{Eidelman69}. 
	\hfill
\end{proof}

\bigskip

The last step is to perform a fixed point argument obtaining a (local in time) solution for the full system (\ref{10})-(\ref{12}).

\begin{theorem}
	\label{Th4}
	Let $U_0\in C^{4+\alpha}[0,\infty)\cap C_0[0,\infty)$ where $\alpha \in (0,1)$, satisfying compatibility conditions
	\begin{equation*}
		U_0= \psi_1 (0,0), \quad \partial_x U_0 = \partial_x^3 U_0 =0 \quad \text{ at } \ x=0,
	\end{equation*}
	and $g$ is sufficiently regular. Then for small $\delta>0$ and small $T$ where $T$ is as in (\ref{18}), there exists $T_0 \in (0,T)$ such that (\ref{10})-(\ref{12}) has a unique solution $(U,\Lambda)\in C^{1+\frac{\alpha}{4}}((0,T_0), C^{4+\alpha}[0,\infty)) \times C^{1+\frac{\alpha}{4}}[0, T_0]$.
	
	The time interval $T_0$ depends on the upper bounds of the H\"older constants of the coefficients of (\ref{10}).
\end{theorem}

\begin{proof}
	As $U$ is given by equation (\ref{14}), in order to satisfy the boundary condition (\ref{12}), one must have
	\begin{equation}
		\label{15}
		\begin{aligned}
			\psi_1(\Lambda,t) &= \int\displaylimits_0^t \int\displaylimits_0^\infty G(0,t,y,\tau) \left[ \dot{\Lambda} \psi_2\; \xi^2_\delta + F_2 + F_3 + F_4 \right](y,\tau)\;\mathrm{d}y \;\mathrm{d}\tau  + \int\displaylimits_0^\infty G(0,t,y,0) U_0(y)\;\mathrm{d}y\\
			&=: I+II+III+IV + V.
		\end{aligned}
	\end{equation}
	The first integral being the most important term to be controlled (all the other terms being small), we try to rewrite and simplify it. Observe that the mass is conserved for the Green function satisfying (\ref{9}). Indeed, by denoting $g(x,t) := \int\displaylimits_0^\infty {G(x,t,y,\tau)\mathrm{d}y}$ and integrating the equations in (\ref{9}) with respect to $y$, we get that $g$ satisfies the same system. Hence by uniqueness,
	\begin{equation*}
		\int\displaylimits_0^\infty {G(x,t,y,\tau)\;\mathrm{d}y} =1, \qquad t\in [0,T].
	\end{equation*}
	Therefore,
	\begin{align}
		I & := \int\displaylimits_0^t \int\displaylimits_0^\infty {G(0,t,y,\tau) (\dot{\Lambda} \psi_2\; \xi^2_\delta) (y,\tau) \;\mathrm{d}y \;\mathrm{d}\tau } \nonumber\\
		& =
		\int\displaylimits_0^t {(\dot{\Lambda} \psi_2) (\tau)} \int\displaylimits_0^\infty {G(0,t,y,\tau) \;\mathrm{d}y \;\mathrm{d}\tau } - \int\displaylimits_0^t \int\displaylimits_0^\infty {G(0,t,y,\tau) (\dot{\Lambda} \psi_2)(\tau) (1-\xi^2_\delta) (y) \;\mathrm{d}y \;\mathrm{d}\tau } \nonumber\\
		& = \int\displaylimits_0^t {\;\dot{\Lambda}(\tau)\; \psi_2 (\Lambda(\tau),\tau) \; \mathrm{d}\tau} - \int\displaylimits_0^t \int\displaylimits_0^\infty {G(0,t,y,\tau) (\dot{\Lambda} \psi_2)(\tau) (1-\xi^2_\delta) (y) \;\mathrm{d}y \;\mathrm{d}\tau } \nonumber\\
		& =: I_1 + I_2. \label{38}
	\end{align}
	At this point, let us denote $\Psi(\Lambda,t) = \int\displaylimits_0^\Lambda {\psi_2(y,t)\;\mathrm{d}y}$, which gives,
	\begin{equation*}
		\frac{\mathrm{d}}{\mathrm{d}t} \Psi = \int\displaylimits_0^\Lambda {\partial_t\psi_2(y,t)\;\mathrm{d}y} + \psi_2(\Lambda(t),t)\;\dot{\Lambda}(t).
	\end{equation*}
	Plugging this into (\ref{38}), we get
	\begin{equation*}
		\begin{aligned}
			I_1 & = \int\displaylimits_0^t {\bigg[\frac{\mathrm{d}}{ \mathrm{d}\tau} \Psi(\Lambda(\tau),\tau) -\int\displaylimits_0^{\Lambda(\tau)} {\partial_\tau\psi_2(y,\tau)\;\mathrm{d}y} \bigg] \mathrm{d}\tau}\\
			& = \left[ \Psi(\Lambda(t),t) - \Psi(0,0)\right] - \int\displaylimits_0^t {\int\displaylimits_0^{\Lambda(\tau)} {\partial_\tau\psi_2(y,\tau)\;\mathrm{d}y} \; \mathrm{d}\tau}.
		\end{aligned}
	\end{equation*}
	Using the definition of $\Psi$ and $\psi_2$, one further obtains, after some integration by parts with respect to $y$,
	\begin{align}
		I_1 & = \int\displaylimits_0^\Lambda {(\partial_y g(y,t)-k)\;\mathrm{d}y}-\int\displaylimits_0^t { (\partial_\tau g(\Lambda(\tau),\tau) - \partial_\tau g (0,\tau)) \; \mathrm{d}\tau} \nonumber\\
		& = \left[ (g(\Lambda,t) - g(0,t))-k\Lambda\right] -\int\displaylimits_0^t { \partial_\tau g(\Lambda(\tau),\tau)\; \mathrm{d}\tau} +\left[g(0,t)-g(0,0)\right] \nonumber\\
		& = g(\Lambda,t) - g(0,0)-k\;\Lambda(t). \label{16}
	\end{align}
	Therefore, combining the relations (\ref{15}), (\ref{38}) and (\ref{16}), one gets, by the definition of $\psi_1$ in (\ref{23}),
	\begin{equation*}
			0 = -g(0,0)-k\Lambda(t)+ I_2 + II + III + IV + V.
	\end{equation*}
	In other words, as $k>0$,
	\begin{align}
		\Lambda(t) &= \frac{1}{k}\left[ -g(0,0)+ I_2 + II + III + IV + V \right] . \label{17}
	\end{align}
	Next we need to estimate the different terms of (\ref{17}), in particular to show that the above relation is a contraction map for $\Lambda$.
	
	\paragraph{Estimate of $I_2$:} Using the estimate (\ref{Green_estimate}) for the fundamental solution, we have
	\begin{align*}
		|I_2| &= \bigg|\int\displaylimits_0^t {\dot{\Lambda}(\tau) \;\psi_2(\Lambda(\tau),\tau) \int\displaylimits_0^\infty {(\xi^2_\delta(y)-1)\;G(0,t,y,\tau)\;\mathrm{d}y}\;\mathrm{d}\tau}\bigg| \\
		&\le c \int\displaylimits_0^t {|\dot{\Lambda}(\tau) \;\psi_2(\Lambda(\tau),\tau)| \int\displaylimits_\delta^\infty {(t-\tau)^{-1/4}\exp\left(-c\;\frac{y}{(t-\tau)^{1/4}} \right) \;\mathrm{d}y}\;\mathrm{d}\tau} \\
		&\le c \int\displaylimits_0^t {|\dot{\Lambda}(\tau) \;\psi_2(\Lambda(\tau),\tau)| \; \exp\left(-c\;\tfrac{\delta}{(t-\tau)^{1/4}} \right) \;\mathrm{d}\tau}
		\; \le c_\delta \int\displaylimits_0^t {|\dot{\Lambda}(\tau)| \;\mathrm{d}\tau}.
	\end{align*}
	Due to the definition of $\psi_2$ in (\ref{23}), it is a bounded quantity given by the solid $g$. The above constant $c_\delta>0$ 
	can be made small by choosing $\delta$ and in turn $T$ suitably which makes the exponential term small enough.
	
	\paragraph{Estimate of $II$:} Recalling the definition of $F_2$ from (\ref{F_2}), as it comprises several terms of the form $\dot{\Lambda}\;\partial^k_x U\; \xi'_\delta, k\in [0,3], $ passing the derivatives onto $G$ by integration by parts in $x$, and using estimate (\ref{Green_estimate}) for the fundamental solution, one obtains the following 
	\begin{align*}
		|II| = \bigg| \int\displaylimits_0^t \int\displaylimits_0^\infty {G(0,t,y,\tau)  F_2(y,\tau)\;\mathrm{d}y \;\mathrm{d}\tau}\bigg| \le c_\delta \;\|U\|_{C((0,T_0),C[0,\infty))} \int\displaylimits_0^t {|\dot{\Lambda}(\tau)| (t-\tau)^{-\beta/4} \;\mathrm{d}\tau},
	\end{align*}
	for some $\beta \in (0,4)$. Here we have used the crucial fact that $\text{supp} (F_2) \subset [\delta, 2\delta]$. As before, choosing $T$ suitably, the constant in the above integral can be made small.
	
	\paragraph{Estimate of $III$:} All the terms in (\ref{F_3}) are small. For example, 
	\begin{align*}
		& \quad \bigg| \int\displaylimits_0^t \int\displaylimits_0^\infty {G(0,t,y,\tau)  \;[\dot{\Lambda} \; 3 x^2 \psi_3\;\xi^2_\delta] (y,\tau) \;\mathrm{d}y \;\mathrm{d}\tau}\bigg| \\
		& \le c\; \delta^2 \int\displaylimits_0^t|\dot{\Lambda}(\tau)| |\psi_3(\Lambda(\tau),\tau)|\int\displaylimits_0^\delta |G(0,t,y,\tau)|\;\mathrm{d}y \; \mathrm{d}\tau  \le c_\delta\; \|\Lambda\|_{C^{1+\frac{\alpha}{4}}[0,T]}.
	\end{align*}
	Similarly, the other two terms can be estimated as,
	\begin{align*}
		& \quad \left| \int\displaylimits_0^t \int\displaylimits_0^\infty {G(0,t,y,\tau)  \;[6(1-\Lambda \xi'_\delta)^4 \;\partial_x ((a\xi_\delta + (1-\xi_\delta) +U)^3 \;\psi_3 \;x^3\;\xi_\delta)  -  \partial_t a\;\xi_\delta] (y,\tau) \;\mathrm{d}y \;\mathrm{d}\tau}\right| \\
		& \le C_\delta(\|\psi_2\|, \|\psi_3\|, \|U(t)\|_{C[0,\infty)}) \;(T+ \|\Lambda\|_{C^{1+\frac{\alpha}{4}}[0,T]}).
	\end{align*}
	Due to the expressions in (\ref{23}), the constant can be chosen small, for suitable $\delta$ and thus $T$.
	
	\paragraph{Estimate of $IV$:} By the same argument as for estimating $I_2$ and also using the fact that $|\partial_x^3 U|\le \delta$ for $|x|$ small as $\partial_x^3U|_{x=0}=0$, we get
	\begin{align*}
		|IV|\le c_\delta (\|U\|_{C((0,T_0),C[0,\infty))})\;(T + \|\Lambda\|_{C^{1+\frac{\alpha}{4}}[0,T]}).
	\end{align*}
	
	
	One can prove the Lipschitz estimates for each term in (\ref{17}) by the same argument as above. Note that the fundamental solution $G$ depends on $\Lambda$ in a Lipschitz way, and hence the integral $V$.
	
	Therefore, the map (\ref{17}) is contractive on $C^{1+\frac{\alpha}{4}}[0,T]$ which yields the existence of the contact point $\Lambda(t)$ satisfying (\ref{12}).
	\hfill
\end{proof}

\bigskip

Finally existence of a unique local strong solution to the  thin film model (\ref{7}) is obtained by going back to the original variable from Theorem \ref{Th4}. Let us denote by $I = [\Lambda, \infty)$ for better readability.

\begin{theorem}
	\label{Th2}
	Let $\Lambda_0, h_0 >0$ and $h_0\in C^{4+\alpha}[\Lambda_0, \infty), \alpha\in (0,1)$ with $h_0\to 1$ as $|x|\to\infty$ satisfying the compatibility conditions
	$$
	h_0= g, \quad \partial_{x} h_0 = \partial_x g -k, \quad h_0 \;\partial_{x}^3 h_0 = -6\Lambda_0\frac{\partial_t g}{ g^2} \qquad \text{ at } x=\Lambda_0, t=0.
	$$
	Also, 	$g\in  C^{1+\frac{\alpha}{4}}((0, T_0], C^{3+\alpha}(I) )$.	Then, there exists $T_0 > 0$ which also depends on the initial data such that (\ref{7}) has a unique strong solution $(h,\Lambda)$, with the regularity
	\begin{equation}
		\label{local_estimate1}
		h\in  C^{1+\frac{\alpha}{4}}((0, T_0], C^{4+\alpha}(I) ), \quad
		\Lambda\in C^{1+\frac{\alpha}{4}}[0, T_0].
	\end{equation}
	Furthermore, the solution depends continuously on the initial data.
\end{theorem}



\section{Stability analysis}
\label{S4}

In this section, we discuss steady state/equilibrium solutions of the model (\ref{7}) and asymptotic stability of the non-linear problem around a steady state.

For the standard thin film equation on an unbounded domain (with slip condition), the conservation of mass usually does not hold. 
Therefore, with our current approach, we consider a periodic setting, or in other words, consecutive solid objects immersed in the liquid. Precisely, we assume that the initial thin film is spatially periodic, i.e. for some $L > 0$,
\begin{equation}
	\label{periodic}
	h_0(x) = h_0(x+2L) \qquad \text{ for all } \ x\in\mathbb{ R}.
\end{equation}
Note that the periodicity assumption is primarily for mathematical simplification
in order to have a bounded domain (and in turn finite mass). It is also possible to consider a fixed lateral boundary, but only with small angle between the free interface and the lateral boundary for the lubrication approximation to be valid. 
Furthermore, it is enough to assume that the configuration is symmetric with respect to $y = L$.
Thus, we add the condition 
$$
\partial_xh =0 \qquad \text{ at } \quad y=L
$$
and hence, problem (\ref{25}).

\subsection{Conservation of mass}
\label{S4.1}
When the solid is stationary i.e. $\partial_t g = 0$, observe that $(h \partial_x^3h)|_{x= \Lambda}= 0$ in $(\ref{25})_2$. Therefore we have, from the system (\ref{25}),
\begin{align*}
	\frac{\mathrm{d}}{\mathrm{d}t} \int\displaylimits_{\Lambda(t)}^L {h(x,t) \;\mathrm{d}x} = \int\displaylimits_{\Lambda(t)}^L {\partial_t h \;\mathrm{d}x} - \dot{\Lambda}\;h|_{x=\Lambda} &= - \int\displaylimits_{\Lambda(t)}^L {\partial_x(h^3 \partial_x^3 h)\;\mathrm{d}x} - \dot{\Lambda}\;h|_{x=\Lambda}\\
	& = - (h^3 \partial_x^3 h)|_{x=\Lambda} - \dot{\Lambda}\;h|_{x=\Lambda} = - \dot{\Lambda}\;h|_{x=\Lambda}
\end{align*}
and
\begin{equation*}
	\frac{\mathrm{d}}{\mathrm{d}t} \int\displaylimits_0^{\Lambda(t)} {g(x,t)\;\mathrm{d}x} = \int\displaylimits_0^{\Lambda(t)} {\partial_t g \;\mathrm{d}x} + \dot{\Lambda} \;g|_{x=\Lambda}
	= \dot{\Lambda}\;g|_{x=\Lambda}
\end{equation*}
which shows due to $(\ref{25})_2$ that the mass is conserved, i.e. 
\begin{equation}
	\label{mass_conv}
	\frac{\mathrm{d}}{\mathrm{d}t}\bigg( \int\displaylimits_0^{\Lambda(t)} {g(x,t)\;\mathrm{d}x} + \int\displaylimits_{\Lambda(t)}^L {h(x,t) \;\mathrm{d}x}\bigg) = \dot{\Lambda}\;g|_{x=\Lambda} - \dot{\Lambda}\;h|_{x=\Lambda}=0.
\end{equation}

\subsection{Equilibrium solution}
In this subsection, we first compute the equilibrium solution of the system (\ref{25}) for a given liquid volume (which happens to be the same as the steady state), as a constrained minimization problem using the Lagrange multipliers method. Subsequently, we deduce the energy dissipation relation.

For a fixed volume of the liquid
\begin{equation}
	\label{volume}
	V_0 = \int\displaylimits_0 ^{\Lambda(t)} {g\;\mathrm{d}x} + \int\displaylimits_{\Lambda(t)}^L { h \;\mathrm{d}x},
\end{equation}
an equilibrium solution $(\overline{h}, \overline{\Lambda})$ of (\ref{25}) is a minimizer of the total energy $E$, given by
\begin{equation}
	\label{energy}
	E(t) := a \int\displaylimits_0^{\Lambda(t)} {|\partial_xg|^2 \mathrm{d}x} + b \int\displaylimits_{\Lambda(t)}^L {|\partial_x h|^2 \mathrm{d}x} + c \int\displaylimits_{\Lambda(t)}^L {|\partial_xg|^2 \mathrm{d}x},
\end{equation}
where $a, b, c$ are non-negative constants with $b>0$, corresponding to the energy for the liquid-solid, liquid-gas and gas-solid interfaces respectively. Note that energy of the solid-gas interface must be included in the total energy. The total energy of the considered system (\ref{25}) comprises of the kinetic energy of the bulk phases and the surface energies of the three interfaces. The former being negligible compared to the latter one, it reduces to (\ref{energy}).

One can compute the new energy for small changes $(\overline{h} +\delta h, \overline{\Lambda} +\delta \Lambda )$ of the equilibrium state as 
\begin{align*}
	& \ \ E + \delta E \\
	&= a \int\displaylimits_0^{\overline{\Lambda} +\delta \Lambda} {|\partial_xg|^2 \mathrm{d}x} + b \int\displaylimits_{\overline{\Lambda} +\delta \Lambda}^L {|\partial_x (\overline{h} +\delta h)|^2 \mathrm{d}x} + c \int\displaylimits_{\overline{\Lambda} +\delta \Lambda}^L {|\partial_xg|^2 \mathrm{d}x}\\
	&= E + (a-c)\delta \Lambda  |\partial_xg|^2|_{x=\overline{\Lambda}} - b\; \delta\Lambda |\partial_x(\overline{h} +\delta h)|^2|_{x = \overline{\Lambda}} + b \int\displaylimits_{\overline{\Lambda}}^L {(2 \partial_x \overline{h} \;\partial_x (\delta h)+ |\partial_x (\delta h)|^2) \;\mathrm{d}x}.
\end{align*}
Neglecting the small quadratic terms, one obtains,
\begin{align*}
	\delta E = (a-c)\;\delta \Lambda |\partial_x g|^2|_{x = \overline{\Lambda}}- b\; \delta\Lambda |\partial_x \overline{h}|^2|_{x = \overline{\Lambda}} + 2b \int\displaylimits_{\overline{\Lambda}}^L { \partial_x \overline{h} \;\partial_x (\delta h) \;\mathrm{d}x}.
\end{align*}
Similarly, the new volume becomes,
\begin{equation*}
	V + \delta V = \int\displaylimits_0 ^{\overline{\Lambda} +\delta \Lambda } {g\;\mathrm{d}x} + \int\displaylimits_{\overline{\Lambda} +\delta \Lambda }^L { (\overline{h} + \delta h) \;\mathrm{d}x} = V + \int\displaylimits_{\overline{\Lambda} }^L { \delta h \;\mathrm{d}x}.
\end{equation*}
Hence,
\begin{equation*}
	\delta V = \int\displaylimits_{\overline{\Lambda} }^L { \delta h \;\mathrm{d}x}.
\end{equation*}
Therefore, we need to find a unique Lagrange multiplier  $\lambda\in \mathbb{R}$ such that $\delta E =\lambda \;\delta V$, i.e. after integrating by parts,
\begin{equation}
	\label{26}
	(a-c)\;\delta \Lambda |\partial_xg|^2|_{x=\overline{\Lambda}} - b (\partial_xg|_{x=\overline{\Lambda}} - k)^2 \delta\Lambda - 2b \int\displaylimits_{\overline{\Lambda}}^L { \partial^2_x \overline{h} \; \delta h  \;\mathrm{d}x} -2b (\partial_x \overline{h} \; \delta h )|_{x=\overline{\Lambda}}= \lambda \int\displaylimits_{\overline{\Lambda} }^L { \delta h \;\mathrm{d}x}.
\end{equation}
Here we used the relation $(\ref{25})_2$ since $(\overline{h}, \overline{\Lambda})$ solves the system (\ref{25}).
Choosing $\delta\Lambda = 0$ and $\delta h$ having compact support, one then obtains,
\begin{equation}
	\label{27}
	- 2b \int\displaylimits_{\overline{\Lambda}}^L { \partial^2_x \overline{h} \;\delta h \;\mathrm{d}x}	= \lambda \int\displaylimits_{\overline{\Lambda} }^L { \delta h \;\mathrm{d}x} \qquad \text{ which implies } \quad \partial^2_x\overline{h} = -\frac{\lambda}{2b}, \quad x\in (\overline{\Lambda}, L).
\end{equation}
Also we have the continuity relation satisfied by the perturbed solution,
\begin{equation*}
	(\overline{h} + \delta h) |_{x = \overline{\Lambda} + \delta \Lambda} = g|_{x = \overline{\Lambda} + \delta \Lambda},
\end{equation*}
which gives, neglecting small terms,
\begin{equation*}
	\overline{h}|_{x=\overline{\Lambda}} + \partial_x \overline{h}|_{x=\overline{\Lambda}} \;\delta\Lambda + \delta h|_{x=\overline{\Lambda}} = g|_{x=\overline{\Lambda}} + \partial_x g|_{x=\overline{\Lambda}}\; \delta \Lambda
\end{equation*}
implying, by using $(\ref{25})_2$,
\begin{equation*}
	\delta h|_{x=\overline{\Lambda}} = (\partial_x g - \partial_x \overline{h})|_{x=\overline{\Lambda}} \;\delta\Lambda = k \;\delta\Lambda.
\end{equation*}
Inserting (\ref{27}) into equation (\ref{26}), we get a relation determining the contact angle, 
\begin{equation}
	\label{Young relation}
	b\; k^2 + |\partial_xg|^2 (a-b-c) =0, \qquad \text{or}, \qquad k = \sqrt{\left( 1-\tfrac{a-c}{b}\right)}\; \partial_xg.
\end{equation}
Here we assume the condition $(b+c-a)/b> 0$ so that the above relation is well defined. Also recall that $k>0$ (by construction).
This is a form of Young’s condition which says that the contact angle between the liquid-gas and the liquid-solid interfaces depends only on the slope of the solid and the energies of the interfaces, in other words, only on the material properties, as mentioned in the Introduction under paragraph "Equation (\ref{0_BC})". 

From relation (\ref{27}) and the boundary conditions $(\ref{25})_2$ together with $\partial_x \overline{h}|_{x=L} = 0$ (due to the symmetry assumption), we can further determine the equilibrium solution completely,
\begin{equation}
	\label{equil_sol}
	\overline{h} = \tfrac{(\partial_x g-k)|_{x=\overline{\Lambda}}}{2(\overline{\Lambda}-L)} (x-L)^2 + g|_{x=\overline{\Lambda}} - \tfrac{(\partial_x g-k)|_{x=\overline{\Lambda}}}{2} (\overline{\Lambda}-L), \qquad x\in (\overline{\Lambda}, L),
\end{equation}
and also the Lagrange multiplier, in terms of $\overline{\Lambda}$, as
\begin{equation*}
	\lambda = 2b \tfrac{(\partial_x g-k)}{(L-\overline{\Lambda})}.
\end{equation*}
Finally, from the volume constraint (\ref{volume}), the contact point at equilibrium $\overline{\Lambda}$ can be determined.

\paragraph{Energy dissipation:}

The main idea in order to perform the stability analysis in Section \ref{S4.5} is to derive the energy dissipation formula. The existence of a classical solution in Theorem \ref{Th2} in case of the half line, which also applies in case of periodic boundary conditions, allows us to make all the following calculations in a rigorous manner.

When the solid is stationary, it holds from $(\ref{25})_2$ that $h|_{x=\Lambda} = g|_{x=\Lambda} $ for $t>0$ and in turn, we obtain the following relation by differentiating in time
\begin{equation}
	\label{28}
	\partial_t h + \dot{\Lambda} \;\partial_x h = \dot{\Lambda}\;\partial_x g \qquad \text{ which implies } \quad \partial_t h = k\dot{\Lambda} \quad \text{ at } x=\Lambda.
\end{equation}
Therefore, one obtains the following energy dissipation relation from (\ref{energy}) as
\begin{align}
	\frac{\mathrm{d}}{\mathrm{d}t}E &= 2b \int\displaylimits_{\Lambda(t)}^L {\partial_x h\; \partial_{xt}h\;\mathrm{d}x} - b \dot{\Lambda}|\partial_x h|^2|_{x=\Lambda} + (a-c)\dot{\Lambda}|\partial_xg|^2|_{x=\Lambda}  \notag\\
	&= -2b \int\displaylimits_{\Lambda(t)}^L {\partial_x h\; \partial^2_{x}(h^3\partial_x^3h)\;\mathrm{d}x} + \dot{\Lambda} ( (a-c)|\partial_xg|^2 -  b|\partial_x h|^2)|_{x=\Lambda} \hspace{1.5cm} [\text{using } (\ref{25})_1] \notag\\
	&= -2b \int\displaylimits_{\Lambda(t)}^L { h^3| \partial^3_{x}h|^2\;\mathrm{d}x}+2b \;(\partial_x h \;\partial_x(h^3\partial_x^3h))|_{x=\Lambda} + \dot{\Lambda} ( (a-c)|\partial_xg|^2 -  b|\partial_x h|^2)|_{x=\Lambda} \notag\\
	&= -2b \int\displaylimits_{\Lambda(t)}^L { h^3| \partial^3_{x}h|^2\;\mathrm{d}x} + \dot{\Lambda} ( (a-c)|\partial_xg|^2 -  b(\partial_xg-k)^2 - 2kb (\partial_xg-k))|_{x=\Lambda} \notag\\
	&= -2b \int\displaylimits_{\Lambda(t)}^L { h^3| \partial^3_{x}h|^2\;\mathrm{d}x} - \dot{\Lambda} ( (b+c-a)|\partial_xg|^2 |_{x=\Lambda} - bk^2)
	\notag\\
	&= -2b \int\displaylimits_{\Lambda(t)}^L { h^3| \partial^3_{x}h|^2\;\mathrm{d}x} . \label{energy_relation}
\end{align}
We have used relations $(\ref{25})_1$, (\ref{28}) and $(\ref{25})_2$ to obtain the fourth equality whereas Young's condition (\ref{Young relation}) is used at the last line. Furthermore, the condition $\partial_x h|_{x=L} = 0$ (due to the symmetry assumption) has been used in the above deduction.

\paragraph{Steady state:}
For the stationary solution of (\ref{25}), as the time derivative vanishes, $(\ref{25})_1$ reduces to
$$
\partial_x(h^3 \partial_x^3h) = 0 \qquad \text{ in } \ (\Lambda,L),
$$
which means two possibilities: either the thin film is given by a constant height or, a parabola (below we give a complete description, cf. Remark \ref{remark_ss}). The constant thin film corresponds to the case where the slope of the thin film is zero, i.e. $\tilde{\theta} = 0$ (cf. Figure \ref{fig4}) which means $\partial_x g = k$ at $x=\Lambda$. If $\partial_x g \neq k$, then the shape of the thin film is given by a parabola. This indicates that the liquid film might rupture in finite time if the volume of the liquid is small.

Recall that the steady state of the classical thin film equation $\partial_t h + \partial_x(h^2 \partial_x^3h)=0$ on a bounded domain (with slip boundary condition) is indeed given by a parabola. 

\begin{remark}
	\label{remark_ss}
	The above energy dissipation relation (\ref{energy_relation}) shows that
	the equilibrium solution $(\overline{h}, \overline{\Lambda})$, given by (\ref{equil_sol}), is also a steady state of (\ref{25}) since it satisfies $\frac{\mathrm{d}}{\mathrm{d}t}E|_{(\overline{h}, \overline{\Lambda})}=0$.
\end{remark}

\begin{remark}
	From the relation (\ref{Young relation}) and $(\ref{25})_2$, one obtains that $\partial_x g -\partial_x\overline{h} = \gamma \partial_x g$ at $x=\overline{\Lambda}$ where $\gamma = \sqrt{(b+c-a)/b}$. Depending on $\gamma\ge 1$ or $\gamma<1$, the steady state has a (locally) convex or concave profile, respectively.
\end{remark}

\subsection{Asymptotic stability of steady state}
\label{S4.5}
Finally we are in a situation to study the stability of the stationary solution $(\overline{h}, \overline{\Lambda})$ of (\ref{25}), given by (\ref{equil_sol}).
To this end, we introduce the following linear transformation
\begin{equation*}
	\overline{x} = \tfrac{x(L-\overline{\Lambda}) + L (\overline{\Lambda}-\Lambda)}{(L-\Lambda)}
\end{equation*}
which maps the domain $(\Lambda, L)$ to the fixed domain $(\overline{\Lambda}, L)$.
Differentiating, we get
$$
\partial_x \overline{x} = \tfrac{(L-\overline{\Lambda})}{(L-\Lambda)}, \qquad \partial_t \overline{x} =- \dot{\Lambda} \tfrac{(L-x)(L-\overline{\Lambda})}{(L-\Lambda)^2}.
$$
In this new coordinate, 
let us denote $H(\overline{x},t) = h(x, t)$. Recall that $g$ is independent of $t$ here since we assume that the solid is stationary.

Now consider a small perturbation of the steady state as $H = \overline{h} + \varphi$ where $|\varphi|\ll 1$. 
First of all, existence of a unique solution of (\ref{25}), for some small time $T_0$, in the periodic setting can be obtained as in Theorem \ref{Th2}. Considering spaces on bounded spatial interval $(\overline{\Lambda}, L)$ instead of $(0,\infty)$ in Theorem \ref{Th4} does not pose any extra difficulty.
Therefore, $\varphi$ satisfies the same regularity as in Theorem \ref{Th2}, for suitable initial data. Note that $\varphi$ has vanishing mean due to the mass conservation property (\ref{mass_conv}),
\begin{equation}
	\label{29}
	\langle \varphi \rangle := \int\displaylimits_{\overline{\Lambda}}^L {\varphi\;\mathrm{d} \overline{x}} =0.
\end{equation}
Inserting the above expression of $H$, the perturbation $\varphi$ satisfies,
\begin{equation}
	\label{37}
	\begin{aligned}
		\partial_{t} \varphi + \left( \tfrac{L-\overline{\Lambda}}{L-\Lambda}\right) ^4\partial_{\overline{x}}( (\overline{h}+\varphi)^3 \partial^3_{\overline{x}} \varphi ) = \dot{\Lambda} \tfrac{(L-\overline{x})}{(L-\Lambda)} (\partial_{\overline{x}}\overline{h} + \partial_{\overline{x}}\varphi) \qquad &\text{ in } \ \overline{x} \in (\overline{\Lambda}, L),\\
		\varphi= \psi_1, \qquad \partial_{\overline{x}} \varphi = \psi_2 , \qquad \partial_{\overline{x}}^3 \varphi = 0 \qquad &\text{ at } \ \overline{x} = \overline{\Lambda},\\
		\partial_{\overline{x}} \varphi =0 \qquad &\text{ at } \ \overline{x} = L,\\
		\varphi = \tilde{h}_0 \ \quad &\text{ at } \ t=0,
	\end{aligned}
\end{equation} 
where $\tilde{h}_0(\overline{x}) \equiv h_0(x)$ and 
\begin{equation*}
	\psi_1(t)\equiv g(\Lambda) - g(\overline{\Lambda}), \qquad \psi_2(t) \equiv \left( \tfrac{L-\Lambda}{L-\overline{\Lambda}}\right)  (\partial_x g(\Lambda)-k) - (\partial_x g(\overline{\Lambda})-k). 
\end{equation*}
Furthermore, with help of the energy dissipation relation for the equilibrium solution,
\begin{equation*}
	\frac{\mathrm{d}}{\mathrm{d}t}E|_{(\overline{h}, \overline{\Lambda})} =	\frac{\mathrm{d}}{\mathrm{d}t}\bigg[a \int\displaylimits_0^{\overline{\Lambda}}{|\partial_{\overline{x}} g|^2\;\mathrm{d}\overline{x}} + b \int\displaylimits_{\overline{\Lambda}}^L {|\partial_{\overline{x}}\overline{h}|^2 \;\mathrm{d}\overline{x}} +c \int\displaylimits_{\overline{\Lambda}}^L{|\partial_{\overline{x}} g|^2\;\mathrm{d}\overline{x}}\bigg] =0,
\end{equation*}	
the energy equation (\ref{energy_relation}) results into, (plugging in the change of variable $\overline{x}$ and the expression of $H$)
\begin{equation}
	\label{energy_linearisation}
	\begin{aligned}
		\frac{\mathrm{d}}{\mathrm{d}t}E_\varphi(t) & := \frac{\mathrm{d}}{\mathrm{d}t}\bigg[b \int\displaylimits_{\overline{\Lambda}}^L {|\partial_{\overline{x}}\varphi|^2 \;\mathrm{d}\overline{x} + \frac{{\widetilde{\Lambda}}^2}{(L-\overline{\Lambda})}(|\partial_{\overline{x}} g|^2|_{\overline{x}=\overline{\Lambda}}-k^2) + O({\widetilde{\Lambda}}^3)} \bigg]\\
		& \ =-2b (1+O(\widetilde{\Lambda})) \int\displaylimits_{\overline{\Lambda}}^L {(\overline{h} + \varphi)^3 |\partial_{\overline{x}}^3 \varphi|^2 \;\mathrm{d}\overline{x}},
	\end{aligned}
\end{equation}
where $ \widetilde{\Lambda} \equiv \widetilde{\Lambda}(t) := (\Lambda(t) -\overline{\Lambda}	)$. Note that in the above expression, $(1+O(\widetilde{\Lambda})) >0$ which is important for subsequent analysis. \\

Next let us prove an priori estimate which is essential to prove the stability result.
The ideas used here, in particular in Lemma \ref{Lem1} and Corollary \ref{Cor1}, are similar to that of \cite[Section 4, 5]{JLN}.

\begin{lemma}
	\label{Lem1}
	Any function $\varphi\in H^3(\overline{\Lambda}, L)$, 
	satisfying the boundary conditions $\partial_x \varphi = \frac{(k_1-k_2)}{k_1(L-\overline{\Lambda})}\varphi$ at $x =\overline{\Lambda}$ and $\partial_x\varphi = 0$ at $x = L$ where $k_1, k_2$ are some non-zero constants, and with zero mean value $\langle\varphi\rangle = 0$,
	satisfies the following estimates,
	\begin{equation}
		\label{30}
		\int\displaylimits_{\overline{\Lambda}}^L {|\partial_x \varphi|^2 \;\mathrm{d}x} \le C \int\displaylimits_{\overline{\Lambda}}^L {|\partial_x^3 \varphi|^2 \;\mathrm{d}x},
	\end{equation}
	and
	\begin{equation}
		\label{31}
		|\varphi(\overline{\Lambda})|^2	\le C \int\displaylimits_{\overline{\Lambda}}^L {|\partial_x^3 \varphi|^2 \;\mathrm{d}x}.
	\end{equation}
\end{lemma}

\begin{remark}
	Note that the boundary condition $\partial_x \varphi = \frac{(k_1-k_2)}{k_1(L-\overline{\Lambda})}\varphi$ at $x =\overline{\Lambda}$ is crucial for estimate (\ref{30}) to hold. It is possible to construct a function (e.g. a quadratic polynomial) with zero mean value and symmetric with respect to $x = L$ whose first derivative does not vanish, that is the right hand side of (\ref{30}) vanishes while not the left hand side. On the other hand, the only function which satisfies all the conditions stated in the above lemma is the null function.
\end{remark}

\begin{proof}
	(i) Let us first note that since $\varphi\in H^1(\overline{\Lambda}, L)$, it is absolutely continuous, in particular, it holds that
	\begin{equation*}
		|\varphi(\overline{\Lambda})|^2 \le C \int\displaylimits_{\overline{\Lambda}}^L {|\partial_x \varphi|^2 \;\mathrm{d}x}.
	\end{equation*}
	Thus, (\ref{31}) is a consequence of (\ref{30}).
	
	(ii) We prove the inequality (\ref{30}) by a contradiction argument. Suppose that for each $m\in \mathbb{N}$, there exists $\varphi_m\in H^3(\overline{\Lambda}, L)$ with the conditions
	\begin{equation*}
		\partial_x \varphi_m |_{x=\overline{\Lambda}} = \frac{\lambda}{2k_2 b}\varphi_m\big|_{x=\overline{\Lambda}}, \qquad \partial_x\varphi_m |_{x=L}=0, \quad \text{ and } \quad \langle\varphi_m\rangle =0,
	\end{equation*}
	such that
	\begin{equation}
		\label{32}
		\int\displaylimits_{\overline{\Lambda}}^L {|\partial_x \varphi_m|^2 \;\mathrm{d}x} \ge m \int\displaylimits_{\overline{\Lambda}}^L {|\partial_x^3 \varphi_m|^2 \;\mathrm{d}x}.
	\end{equation}
	Further we may renormalize the sequence as $\|\partial_x \varphi_m\|_{L^2(\overline{\Lambda}, L)} = 1$.
	
	Relation (\ref{32}) shows that $\{\varphi_m\}$ is a bounded sequence in $H^3(\overline{\Lambda}, L)$. In fact, $\partial_x^3 \varphi_m \to 0$ in $L^2(\overline{\Lambda}, L)$. Thus, there exists a subsequence, still denoted by $\{\varphi_m\}$, and a function $\varphi$ such that $\varphi_m \rightharpoonup \varphi$ in $H^3(\overline{\Lambda}, L)$. Due to the compactness of $H^3(\overline{\Lambda}, L)\hookrightarrow H^2(\overline{\Lambda}, L)$, one further has that $\varphi_m \to \varphi$ strongly in $H^2(\overline{\Lambda}, L)$. Therefore, $\varphi$ satisfies,
	\begin{equation}
		\label{33}
		\partial_x \varphi |_{x=\overline{\Lambda}} = \frac{\lambda}{2k_2 b}\varphi\big|_{x=\overline{\Lambda}}, \qquad \partial_x\varphi |_{x=L}=0, \quad \quad \langle\varphi\rangle =0,
	\end{equation}
	and
	\begin{equation}
		\label{34}
		\|\partial_x \varphi_m\|_{L^2(\overline{\Lambda}, L)} =1.
	\end{equation}
	On the other hand, one has from (\ref{32}) that $\partial_x^3 \varphi = 0$ in $L^2(\overline{\Lambda}, L)$. This implies, together with the boundary conditions and the vanishing mean (\ref{33}), that $\varphi\equiv 0$ in $ (\overline{\Lambda}, L)$ which is a contradiction to (\ref{34}).
	\hfill
\end{proof}

\bigskip

Now we can establish the asymptotic stability of the energy corresponding to (\ref{25}).

\begin{corollary}
	\label{Cor1}	
	Let $(\overline{h}, \overline{\Lambda})$ be a steady state of (\ref{25}), given by (\ref{equil_sol}).
	There exists $\varepsilon>0$ such that for any initial data $\tilde{h}_0 \in C^{4+\alpha}(\overline{\Lambda}, L), \alpha\in (0,1)$ and $\Lambda_0 \in\mathbb{R}$ with $|\Lambda_0-\overline{\Lambda}| + \|\tilde{h}_0 - \overline{h}\|_{C^1 (\overline{\Lambda}, L)}\le \varepsilon $, a solution $(h, \Lambda) $ of (\ref{25}) satisfies, for  $t\in(0,T_0)$,
	\begin{equation}
		\label{22}
		|\Lambda-\overline{\Lambda}|\le Ce^{-\omega t} \qquad \text{ and } \qquad \|H- \overline{h}\|_{H^1 (\overline{\Lambda}, L)}\le Ce^{-\omega t} \quad \text{ for some } \omega>0.
	\end{equation}
\end{corollary}

\begin{proof}
	The right hand side of (\ref{energy_linearisation}) can be estimated, since $|\varphi|\ll 1$, as
	\begin{equation*}
		2b (1+O(\widetilde{\Lambda})) \int\displaylimits_{\overline{\Lambda}}^L {(\overline{h} + \varphi)^3 |\partial_{\overline{x}}^3 \varphi|^2 \;\mathrm{d}\overline{x}} \ge C \int\displaylimits_{\overline{\Lambda}}^L {\left( \frac{\overline{h} }{2}\right) ^3 |\partial_{\overline{x}}^3 \varphi|^2 \;\mathrm{d}\overline{x}} \ge C \;\min_{\overline{x}} \overline{h}\;\int\displaylimits_{\overline{\Lambda}}^L { |\partial_{\overline{x}}^3 \varphi|^2 \;\mathrm{d}\overline{x}}.
	\end{equation*}
	Recall that in the above expression, $(1+O(\widetilde{\Lambda})) >0$ as obtained in the deduction of (\ref{energy_linearisation}).	Combining with the estimates in Lemma \ref{Lem1}, we get that
	\begin{equation}
		\label{35}
		2b (1+O(\widetilde{\Lambda})) \int\displaylimits_{\overline{\Lambda}}^L {(\overline{h} + \varphi)^3 |\partial_{\overline{x}}^3 \varphi|^2 \mathrm{d}\overline{x}} \ge C E_\varphi(t),
	\end{equation}
	where the energy for the perturbation $E_\varphi$ is defined in (\ref{energy_linearisation}). Therefore (\ref{35}) yields that
	\begin{equation*}
		\frac{\mathrm{d}}{\mathrm{d}t}E_\varphi \le -C E_\varphi \qquad \text{ for } \quad 0\le t\le T_0
	\end{equation*}
	which finally gives,
	\begin{equation*}
		E_\varphi(t) \le \exp(-Ct )E_\varphi(0) \qquad \text{ for } \quad 0\le t\le T_0.
	\end{equation*}
	In particular, one gets for the full energy defined by (\ref{energy}),
	\begin{equation*}
		E(t) \le \exp(-Ct )E(0) \qquad \text{ for } \quad 0\le t\le T^*.
	\end{equation*}
	This completes the proof.
	\hfill
\end{proof}
\bigskip

Next, 
to obtain the global in time existence result, we follow the path as in Section \ref{S3.3}. To begin with, a suitable estimate on the fundamental solution for the linear problem corresponding to (\ref{37}) can be obtained as before. Here onwards, the bar over $x$ is omitted.

\begin{lemma}
	\label{Lem2}
	Let $\tilde{h}_0\in C^{4+\alpha}(\overline{\Lambda},L)$ for $\alpha\in (0,1)$.		There exists a fundamental solution $G(x,t,\xi,\tau)$ of the linear problem
	\begin{equation}
		\label{36}
		\begin{aligned}
			\partial_{t} \varphi + \left( \tfrac{L-\overline{\Lambda}}{L-\Lambda}\right) ^4\partial_{x}( \overline{h}^3 \partial^3_{x} \varphi ) - \dot{\Lambda} \tfrac{(L-x)}{(L-\Lambda)} \partial_{x}\varphi =0 \qquad \ &\text{ in } \ x \in (\overline{\Lambda}, L), t>0,\\
			\partial_{x} \varphi = \psi_2, \qquad \partial_{x}^3 \varphi = 0 \qquad \ &\text{ at } \ x = \overline{\Lambda}, t>0,\\
			\partial_{x} \varphi =0 \qquad \ &\text{ at } \ x = L, t>0,\\
			\varphi = \tilde{h}_0 \ \qquad &\text{ at } \ t=0, x \in (\overline{\Lambda}, L),
		\end{aligned}
	\end{equation}
	satisfying the estimate
	\begin{equation}
		\label{21}
		|\partial^m G(x,t, \xi,\tau)|\le c_m (t-\tau)^{-\frac{m+1}{4}} \exp\left\lbrace -c \left( \tfrac{|x-y|}{(t-\tau)^{1/4}}\right) ^{4/3} \right\rbrace , \quad t\in [0,T] .
	\end{equation}
\end{lemma}

\begin{proof}
	The result follows from \cite[Chapter IV.2, Theorem 3.4]{Eidelman69}, as in Lemma \ref{Lem4}.
	As per the notations in \cite{Eidelman69}, for our system (\ref{36}), $N=1=n, b=2, r_1 =1, r_2 =3$ and the boundary conditions read as $B\equiv (\partial_{x}, \partial_{x}^3)u|_{x=\overline{\Lambda}}=f\equiv (\psi_2, 0)$. In $(\ref{36})_1$, the leading order coefficient $a_4(x,t) = \left( \frac{L-\overline{\Lambda}}{L-\Lambda}\right) ^4 \overline{h}^3$ is H\"older continuous in both $t\in (0,T_0)$ and $x\in (0,L)$, since Theorem \ref{Th2} gives that $\Lambda \in C^{1+\frac{\alpha}{4}}[0,T_0]$ and $\overline{h}$ is smooth. Furthermore, the lower order coefficients $a_3(x,t) = \left( \frac{L-\overline{\Lambda}}{L-\Lambda}\right) ^4 \partial_x(\overline{h}^3)$ and $a_1(x,t) =\dot{\Lambda} \left( \frac{L-x}{L-\Lambda}\right)$ are H\"older continuous in $x$ as well. Hence, all conditions of \cite[Chapter IV.2, Theorem 3.4]{Eidelman69} are satisfied.
	\hfill
\end{proof}

\begin{theorem}
	\label{Th3}
	Let $(\overline{h}, \overline{\Lambda})$ be a steady state of (\ref{25}) and $\tilde{h}_0 \in C^{4+\alpha}(\overline{\Lambda}, L), \alpha\in (0,1)$ and $\Lambda_0 \in\mathbb{R}$. There exists $\varepsilon>0$ such that for any initial data satisfying $|\Lambda_0-\overline{\Lambda}| + \|\tilde{h}_0 - \overline{h}\|_{C^{4+\alpha} (\overline{\Lambda}, L)}\le \varepsilon $, a unique global solution $(h, \Lambda) $ of the problem (\ref{25}) exists, with regularity given by Theorem \ref{Th2}, and satisfies, for all $t>0$,
	\begin{equation*}
		|\Lambda-\overline{\Lambda}|\le Ce^{-\omega t} \qquad \text{ and } \qquad \|H- \overline{h}\|_{C^{4+\alpha} (\overline{\Lambda}, L)}\le Ce^{-\omega t} \quad \text{ for some } \omega>0.
	\end{equation*}
\end{theorem}

\begin{proof}
	Theorem \ref{Th2} provides existence of $\varphi$ as a unique solution of the quasilinear problem
	\begin{equation}
		\label{37.}
		\begin{aligned}
			\partial_{t} \varphi + \left( \tfrac{L-\overline{\Lambda}}{L-\Lambda}\right) ^4\partial_{x}( (\overline{h}+\varphi)^3 \partial^3_{x} \varphi ) = \dot{\Lambda} \tfrac{(L-x)}{(L-\Lambda)} (\partial_{x}\overline{h} + \partial_{x}\varphi) \qquad \ &\text{ in } \ x \in (\overline{\Lambda}, L), t>0,\\
			\partial_{x} \varphi = \psi_2, \quad \partial_{x}^3 \varphi = 0 \qquad \ &\text{ at } \ x = \overline{\Lambda}, t>0,\\
			\partial_{x} \varphi =0 \qquad \ &\text{ at } \ x = L, t>0,\\
			\varphi = \tilde{h}_0 \qquad &\text{ at } \ t=0, x \in (\overline{\Lambda}, L),
		\end{aligned}
	\end{equation}
	in some interval $[0,T_0]$, which in addition matches the boundary condition $\varphi|_{x=\overline{\Lambda}} = \psi_1$, and is given by the following integro-differential form,
	\begin{equation}
		\label{20}
		\varphi(x,t) = \int\displaylimits_0^t \int\displaylimits_{\overline{\Lambda}}^L G(x,t,y,\tau) \left[ F_1 + F_2 \right](y,\tau)\;\mathrm{d}y \;\mathrm{d}\tau.
	\end{equation}
	Here $G$ is the fundamental solution of the linear problem given in Lemma \ref{Lem2} and
	\begin{equation*}
		F_1 = \left( \tfrac{L-\overline{\Lambda}}{L-\Lambda}\right) ^4 \partial_x ((\varphi^3 + 3\overline{h}\; \varphi\;(\overline{h}+ \varphi))\;\partial_x^3 \varphi), \qquad F_2 = \dot{\Lambda}\; \tfrac{(L-x)}{(L-\Lambda)} \;\partial_{x}\overline{h}.
	\end{equation*}
	
	Since the assumptions of the local existence result (eg. \cite[Chapter III.4, Theorem 8.3]{Eidelman69}) are satisfied, there exists a unique solution $\varphi_1$ of the Cauchy problem $\varphi_1|_{t=T_0} = \varphi(x, T_0)$ for equation (\ref{37.}), which is defined for $t\in (T_0, T_1)$ where $T_1>T_0$ and belongs to $C^{4+\alpha}(\overline{\Lambda},L)$. In turn, also $\Lambda\in C^{1+\frac{\alpha}{4}}[T_0, T_1]$. The solution can be continued in this way.
	
	Now observe that, 
	\begin{align*}
		I & = \bigg| \int\displaylimits_0^t \int\displaylimits_{\overline{\Lambda}}^L G(x,t,y,\tau) F_1 (y,\tau)\;\mathrm{d}y \;\mathrm{d}\tau\bigg| \\
		& = \bigg| \int\displaylimits_0^t {\left( \tfrac{L-\overline{\Lambda}}{L-\Lambda}\right) ^4(\tau)} \int\displaylimits_{\overline{\Lambda}}^L G(x,t,y,\tau) \partial_x ((\varphi^3 + 3\overline{h}\; \varphi\;(\overline{h}+ \varphi))\;\partial_x^3 \varphi) (y,\tau)\;\mathrm{d}y \;\mathrm{d}\tau\bigg| \\
		& = \bigg| \int\displaylimits_0^t {\left( \tfrac{L-\overline{\Lambda}}{L-\Lambda}\right) ^4(\tau)} \int\displaylimits_{\overline{\Lambda}}^L \partial_x G(x,t,y,\tau) (\varphi^3 + 3\overline{h}\; \varphi\;(\overline{h}+ \varphi))\;\partial_x^3 \varphi (y,\tau)\;\mathrm{d}y \;\mathrm{d}\tau\bigg|\\
		& \le c(\|\varphi\|_{C((0,t),C^{\frac{1}{2}}(\overline{\Lambda},L))})
		\int\displaylimits_0^t {\left( \tfrac{L-\overline{\Lambda}}{L-\Lambda}\right) ^4(\tau)} \int\displaylimits_{\overline{\Lambda}}^L |\partial_x G(x,t,y,\tau)| \;\mathrm{d}y \;\mathrm{d}\tau\\
		& \le c(\|\varphi\|_{C((0,t),C^{\frac{1}{2}}(\overline{\Lambda},L))})
		\int\displaylimits_0^t {\left( \tfrac{L-\overline{\Lambda}}{L-\Lambda}\right) ^4(\tau) (t-\tau)^{-\frac{1}{4}}\;\mathrm{d}\tau}.
	\end{align*}
	In the above we used integration by parts to pass all the derivatives over $G$ and then use estimate (\ref{21}) for $m=1$ and the regularity of $\overline{h}$ and $\varphi$. Similarly, we have
	\begin{align*}
		II & = \bigg| \int\displaylimits_0^t \int\displaylimits_{\overline{\Lambda}}^L G(x,t,y,\tau) F_2 (y,\tau)\;\mathrm{d}y \;\mathrm{d}\tau\bigg| \\
		& = \bigg| \int\displaylimits_0^t {\tfrac{\dot{\Lambda}}{(L-\Lambda)}(\tau)} \int\displaylimits_{\overline{\Lambda}}^L G(x,t,y,\tau) \; (L-y) \;\partial_{x}\overline{h} (y)\;\mathrm{d}y \;\mathrm{d}\tau\bigg| \\
		& \le c \int\displaylimits_0^t {\left| \dot{\Lambda}\tfrac{(L-\overline{\Lambda})}{(L-\Lambda)}(\tau)\right| } \int\displaylimits_{\overline{\Lambda}}^L |G(x,t,y,\tau)| \;\mathrm{d}y \;\mathrm{d}\tau\\
		& \le c \int\displaylimits_0^t { |\dot{\Lambda}(\tau)|\tfrac{(L-\overline{\Lambda})}{(L-\Lambda)} \;\mathrm{d}\tau} \le c \int\displaylimits_0^t { |\dot{\Lambda}(\tau)| (1 + \tfrac{\widetilde{\Lambda}}{(L-\overline{\Lambda})}+ O({\widetilde{\Lambda}}^2)) \;\mathrm{d}\tau}.
	\end{align*}
	Therefore, from expression (\ref{20}), we can conclude
	\begin{equation*}
		\|\varphi\|_{C^{4+\alpha}(\overline{\Lambda},L)} \le C \left( |\widetilde{\Lambda}(t)| + \|\varphi\|_{C^{\frac{1}{2}} (\overline{\Lambda}, L)}\right)
		\le C \left( |\Lambda-\overline{\Lambda}| + \|\varphi\|_{H^1 (\overline{\Lambda}, L)}\right).
	\end{equation*}
	Thus, the local solution can be continued up to time $T=\infty$ in the solution space, as long as the initial data stays close enough. 
	This concludes the proof.
	\hfill
\end{proof}

\bigskip

\textbf{Acknowledgements.} The authors gratefully acknowledge the financial support of the
collaborative research centre "The mathematics of emerging effects" (CRC 1060,
Project-ID 211504053) and the Hausdorff Center for Mathematics (EXC 2047/1,
Project-ID 390685813), both of them funded through the Deutsche
Forschungsgemeinschaft (DFG, German Research Foundation). 
\\

The authors certify that they do not have any conflict of interest.

{\small

\begin{thebibliography}{10}
		
		\bibitem{BL22}
		G.~Beck and D.~Lannes.
		\newblock Freely floating objects on a fluid governed by the {B}oussinesq
		equations.
		\newblock {\em Ann. Inst. H. Poincar\'{e} C Anal. Non Lin\'{e}aire},
		39(3):575--646, 2022.
		
		\bibitem{BF90}
		F.~Bernis and A.~Friedman.
		\newblock Higher order nonlinear degenerate parabolic equations.
		\newblock {\em Journal of Differential Equations}, 83(1):179--206, 1990.
		
		\bibitem{Bertozzi98}
		A.~L. Bertozzi.
		\newblock The mathematics of moving contact lines in thin liquid films, 1998.
		
		\bibitem{CM12}
		L.~Chupin and S.~Martin.
		\newblock Rigorous derivation of the thin film approximation with
		roughness-induced correctors.
		\newblock {\em SIAM J. Math. Anal.}, 44(4):3041--3070, 2012.
		
		\bibitem{DP20}
		A.L. Dalibard and C.~Perrin.
		\newblock Existence and stability of partially congested propagation fronts in
		a one-dimensional {N}avier-{S}tokes model.
		\newblock {\em Commun. Math. Sci.}, 18(7):1775--1813, 2020.
		
		\bibitem{degennes}
		P.~G. de~Gennes.
		\newblock Wetting: statics and dynamics.
		\newblock {\em Rev. Mod. Phys.}, 57:827--863, Jul 1985.
		
		\bibitem{Eidelman69}
		S.~D. Eidel{'}man.
		\newblock {\em Parabolic systems}.
		\newblock North-Holland Publishing Co., Amsterdam-London; Wolters-Noordhoff
		Publishing, Groningen, 1969.
		\newblock Translated from the Russian by Scripta Technica, London.
		
		\bibitem{GNV}
		A.~Ghosh, B.~Niethammer, and J.~J.~L. Vel\'azquez.
		\newblock Revisiting {S}hikhmurzaev's approach to the contact line problem.
		\newblock {\em Acta Applicandae Mathematicae}, 181, 2022.
		
		\bibitem{GGKO14}
		L.~Giacomelli, M.~V. Gnann, H.~Knüpfer, and F.~Otto.
		\newblock Well-posedness for the {N}avier-slip thin-film equation in the case
		of complete wetting.
		\newblock {\em Journal of Differential Equations}, 257(1):15--81, 2014.
		
		\bibitem{GKO08}
		L.~Giacomelli, H.~Knüpfer, and F.~Otto.
		\newblock Smooth zero-contact-angle solutions to a thin-film equation around
		the steady state.
		\newblock {\em Journal of Differential Equations}, 245(6):1454--1506, 2008.
		
		\bibitem{gnann18}
		M.~V. Gnann and M.~Petrache.
		\newblock The {N}avier-slip thin-film equation for 3d fluid films: Existence
		and uniqueness.
		\newblock {\em Journal of Differential Equations}, 265(11):5832--5958, 2018.
		
		\bibitem{GPSW20}
		E.~Godlewski, M.~Parisot, J.~Sainte-Marie, and F.~Wahl.
		\newblock Congested shallow water model: on floating body.
		\newblock {\em SMAI J. Comput. Math.}, 6:227--251, 2020.
		
		\bibitem{GT}
		Y.~Guo and I.~Tice.
		\newblock Stability of contact lines in fluids: 2{D} {S}tokes flow.
		\newblock {\em Archive for Rational Mechanics and Analysis}, 227(2):767--854,
		2018.
		
		\bibitem{hocking92}
		L.~M. Hocking.
		\newblock Rival contact-angle models and the spreading of drops.
		\newblock {\em Journal of Fluid Mechanics}, 239:671–681, 1992.
		
		\bibitem{HS}
		C.~Huh and L.E. Scriven.
		\newblock Hydrodynamic model of steady movement of a solid/liquid/fluid contact
		line.
		\newblock {\em Journal of Colloid and Interface Science}, 35(1):85 -- 101,
		1971.
		
		\bibitem{IL21}
		T.~Iguchi and D.~Lannes.
		\newblock Hyperbolic free boundary problems and applications to wave-structure
		interactions.
		\newblock {\em Indiana Univ. Math. J.}, 70(1):353--464, 2021.
		
		\bibitem{JLN}
		J.~Jansen, C.~Lienstromberg, and K.~Nik.
		\newblock Long-time behavior and stability for quasilinear doubly degenerate
		parabolic equations of higher order.
		\newblock {\em SIAM Journal on Mathematical Analysis}, 55(2):674--700, 2023.
		
		\bibitem{hans11}
		H.~Knüpfer.
		\newblock Well-posedness for the {N}avier slip thin-film equation in the case
		of partial wetting.
		\newblock {\em Communications on Pure and Applied Mathematics},
		64(9):1263--1296, 2011.
		
		\bibitem{KM15}
		H.~Knüpfer and N.~Masmoudi.
		\newblock {Darcy’s Flow with Prescribed Contact Angle: Well-Posedness and
			Lubrication Approximation}.
		\newblock {\em Archive for Rational Mechanics and Analysis}, 218:589 -- 646,
		2015.
		
		\bibitem{MR15}
		A.~Munnier and K.~Ramdani.
		\newblock Asymptotic analysis of a neumann problem in a domain with cusp.
		application to the collision problem of rigid bodies in a perfect fluid.
		\newblock {\em SIAM Journal on Mathematical Analysis}, 47(6):4360--4403, 2015.
		
		\bibitem{navier}
		C.~L. M.~H. Navier.
		\newblock M\'{e}moire sur les lois du mouvement des fluides.
		\newblock {\em M\'{e}m. Acad. Sci. Inst. de France (2)}, pages 389--440, 1823.
		
		\bibitem{ODB}
		A.~Oron, S.~H. Davis, and S.~G. Bankoff.
		\newblock Long-scale evolution of thin liquid films.
		\newblock {\em Reviews of Modern Physics}, 69:931--980, 1997.
		
		\bibitem{RenE11}
		W.~Ren and W.~E.
		\newblock Derivation of continuum models for the moving contact line problem
		based on thermodynamic principles.
		\newblock {\em Communications in Mathematical Sciences}, 9:597--606, 06 2011.
		
		\bibitem{seis18}
		C.~Seis.
		\newblock {The thin-film equation close to self-similarity}.
		\newblock {\em Analysis \& PDE}, 11(5):1303 -- 1342, 2018.
		
		\bibitem{shikh2020}
		Y.~D. Shikhmurzaev.
		\newblock Moving contact lines and dynamic contact angles: a ‘litmus test’
		for mathematical models, accomplishments and new challenges.
		\newblock {\em The European Physical Journal Special Topics},
		229(10):1945--1977, 2020.
		
		\bibitem{Solonnikov77}
		V.~A. Solonnikov.
		\newblock Solvability of a problem on the motion of a viscous incompressible
		fluid bounded by a free surface.
		\newblock {\em Mathematics of the {USSR}-Izvestiya}, 11(6):1323--1358, dec
		1977.
		
		\bibitem{solonnikov95}
		V.~A. Solonnikov.
		\newblock On some free boundary problems for the {N}avier-{S}tokes equations
		with moving contact points and lines.
		\newblock {\em Mathematische Annalen}, 302:743–772, 1995.
		
	\end{thebibliography}

}
\end{document}